\documentclass{birkjour}
\usepackage{amsmath, amssymb, amsthm}

\newcommand{\R}{{\mathbf R}}
\newcommand{\Z}{{\mathbf Z}}
\newcommand{\N}{{\mathbf N}}

\newcommand    {\e}{{\mathbf x}}
\newcommand    {\by}{{\mathbf y}}
\newcommand    {\bt}{{\mathbf t}}

\newcommand    {\bs}{{\mathbf s}}
\newcommand    {\bu}{{\mathbf u}}
\newcommand    {\be}{{\mathbf e}}
\newcommand    {\C}{{\mathbf C}}

\newcommand    {\bz}{{\mathbf z}}

\newcommand    {\fb}{{\mathbf f}}

\newcommand    {\m}{{\rm sign}}

\newtheorem{theorem}{Theorem}[section]
\newtheorem{proposition}[theorem]{Proposition}
\newtheorem{lemma}[theorem]{Lemma}
\newtheorem{corollary}[theorem]{Corollary}
\theoremstyle{definition}
\newtheorem{definition}[theorem]{Definition}
\theoremstyle{remark}
\newtheorem{remark}[theorem]{Remark}
\numberwithin{equation}{section}
\newtheorem{example}{Example}
\newtheorem{problem}{\bf Problem}[section]

\renewcommand{\theequation}{\thesection.\arabic{equation}}

\begin{document}

\title[Discreteness of the spectrum for Schr\"odinger operator]{
	Conditions for discreteness of the spectrum to
	 multi-dimensional Schr\"odinger operator}

\author[L. Zelenko]
{ Leonid Zelenko} 

\address{%
Department of Mathematics \\
University of Haifa  \\
Haifa 31905  \\
Israel}
\email{zelenko@math.haifa.ac.il}

\begin{abstract}
This work is a continuation of our previos paper \cite{Zel1}, where
for the the Schr\"odinger operator $H=-\Delta+ V(\e)\cdot$ $(V(\e)\ge 0)$,  acting in the space $L_2(\R^d)\,(d\ge 3)$, some constructive sufficient conditions for discreteness of its spectrum have been
obtained on the base of  well known Mazya -Shubin criterion and an optimization problem for a set function. Using a {\it capacitary strong type inequality} of David Adams, the concept of {\it base polyhedron} for the harmonic capacity and some properties of Choquet integral by this capacity, we obtain more general sufficient conditions for discreteness of the spectrum of $H$ in terms of a repeated nonincreasing rearrangement of the function $Y(\e,\bt)=\sqrt{V(\e)}\frac{1}{|\e-\bt|^{d-2}}\sqrt{V(\bt)}$ on cubes that are going to infinity.    	
\end{abstract}

\subjclass{Primary 47F05, 47B25, 47D08, \\35P05; Secondary 81Q10, 90C10, 90C27, 91A12} 	

\keywords{Schr\"odinger operator,
	discreteness of the spectrum,  base polyhedron of a submodular set function, Choquet integral,  capacitary inequalities, rearrangement of a function.} 

\maketitle

\tableofcontents

\section{Introduction} \label{sec:introduction}
\setcounter{equation}{0}

This work is a continuation of our previos paper \cite{Zel1}.
 We consider the Schr\"odinger operator $H=-\Delta+ V(\e)\cdot$,
acting in the space $L_2(\R^d)$. In what follows we
assume that $d\ge 3$, $V(\e)\ge 0$ and $V(\cdot)\in L_{1,loc}(\R^d)$. Physically $V(\e)$ is the 
potential of an external electric field. In \cite{Zel1} some constructive sufficient conditions for discreteness of the spectrum of $H$ have been
obtained on the base of  well known Mazya -Shubin criterion (\cite{M-Sh}) and an optimization problem for a set function. Since also in the present paper we shall use  the Mazya -Shubin result, let us formulate it. 
Following to \cite{M-Sh}, 
consider in $\R^d$ an open domain $\mathcal{G}$ satisfying the conditions:

(a) $\mathcal{G}$ is bounded and star-shaped with respect to any point of an open ball $B_\rho(0)\,(\rho>0)$ contained in $\mathcal{G}$;

(b) $\mathrm{diam}(\mathcal{G})=2$.

As it was noticed in \cite{M-Sh}, condition (a) implies that $\mathcal{G}$ can be represented in the form
\begin{equation}\label{formofcalG}
\mathcal{G}=\{\e\in\R^d:\,\e=r\omega,\, |\omega|=1,\,0\le r<r(\omega)\}, 
\end{equation} 
where $r(\omega)$ is a positive Lipschitz function on the standard unit sphere $S^{d-1}\subset\R^d$. For $r>0$
and $\by\in\R^d$ denote 
\begin{equation}\label{dfcalGry}
\mathcal{G}_r(\by):=\{\e\in\R^d:\,r^{-1}\e\in\mathcal{G}\}+\by.
\end{equation}
Denote by
$\mathcal{N}_{\gamma}(\by,r)\;(\gamma\in(0,1))$ the set of all
compact sets $F\subseteq\bar{\mathcal{G}}_r(\by)$ satisfying the condition
\begin{equation}\label{defNcap}
\mathrm{cap}(F)\le\gamma\,\mathrm{cap}(\bar{\mathcal G}_r(\by)),
\end{equation}
where $\mathrm{cap}(F)$ is the harmonic capacity.
\begin{theorem}\label{thMazSh}[\cite{M-Sh}, Theorem 2,2]
	The spectrum of the operator $H$ is discrete, if  
	for some $r_0>0$ and for any $r\in(0,r_0)$ the condition
	\begin{equation}\label{cndmolch1}
	\lim_{\by\rightarrow\infty}\inf_{F\in\mathcal{N}_{\gamma(r)}(\by,r)}\int_{{\mathcal G}_r(\by)\setminus
		F}V(\e)\, \mathrm{d}\e=\infty,
	\end{equation}
	is satisfied, where 
	\begin{equation}\label{cndgammar}
	\forall\,r\in(0,\,r_0):\;\gamma(r)\in(0,1)\quad\mathrm{and}\quad
	\limsup_{r\downarrow 0} r^{-2}\gamma(r)=\infty,
	\end{equation}	
\end{theorem}	
In \cite{M-Sh} also a necessary condition for discreteness of the spectrum was obtained, which is close to sufficient one.

As we have noticed in \cite{Zel1}, condition \eqref{cndmolch1} of Theorem \ref{thMazSh} is hardly verifiable, because in order to test it, one needs to solve a difficult optimization problem, whose cost functional is the set function  $\mathcal{I}(F)=\int_{{\mathcal G}_r(\by)\setminus F}V(\e)\, \mathrm{d}\e$ and the constrain 
$F\in\mathcal{N}_{\gamma(r)}(\by,r)$ is submodular (because``cap'' is a submodular set function (definition \eqref{submod})). In the papers \cite{Ben-Fort}, \cite{Si1}, \cite{L-S-W} and \cite{GMD} some constructive sufficient conditions for discreteness of the spectrum for $H$ have been found without use of the Mazya -Shubin result. In \cite{Zel1} we have estimated the cost functional $\mathcal{I}(F)$ in \eqref{cndmolch1}  from below using the isocapacity inequality and replacing $F\in\mathcal{N}_{\gamma(r)}(\by,r)$ by a weaker but additive constrain. To this end we also used the concept of {\it base polyhedron} for the harmonic capacity (definition \eqref{defcore}). By this way on the base of  Theorem \ref{thMazSh} we have obtained in \cite{Zel1} some constructive sufficient conditions for discreteness of the spectrum in terms of measures, which permit a reformulation in terms of non-increasing rearrangements of some functions connected with the potential $V(\e)$. As we have shown, these conditions are more general than ones obtained in the papers mentioned above.

In the present paper we have obtained more general than in \cite{Zel1} constructive sufficient conditions
for discreteness of the spectrum of $H$, using along with the arguments mentioned above also a
{\it capacitary strong type inequality} \eqref{strtypineq} of David Adams \cite{AH}
 and some properties of Choquet integral by this capacity. 
 
 Let us notice that in \cite{T} Michael Taylor have found an alternative for the Mazya-Shubin result. His necessary and sufficient conditions for discreteness of the spectrum of $H$ are formulated in terms of the {\it scattering length } of the potential $V(\e)$ on boxes that are going to infinity. It would be interesting to extract from this result an easier verifiable sufficient condition which would be more general than results obtained in the present paper. 
 
 Let us describe briefly the main results of this paper. 
 
 Theorem \ref{thcondlebesgue} yields a sufficient condition for discreteness of the spectrum of $H$ in terms of the non-increasing rearrangement of $V(\e)$  with respect to Lebesgue measure on cubes that are  going to infinity. Its proof is based immediately on Corollary 3.12 of our previous work \cite{Zel1}.
 
 In Proposition \ref{threlprlogdenstlebes} we compare Theorem \ref{thcondlebesgue} and Theorem 3.14 of the previous work.
 
 A central role in our considerations plays Theorem \ref{thusecore}, were on the base of Mazya-Shubin result and  inequality \eqref{strtypineq}, mentioned above, we obtain a sufficient condition for discreteness of the spectrum of $H$ in terms of measures from the base polyhedron $BP(\mathcal{G}_r(\by))$of harmonic capacity and the composition of Bessel kernel $G_1(\e)$ with the function $\sqrt{V(\e)}$ on subsets of $\mathcal{G}_r(\by)$ whose complements are small with respect to these measures.
 
 Theorem \ref{thdoublrear}, based on Theorem \ref{thusecore}, yields a sufficient condition for discreteness of the spectrum of $H$ in terms of the repeated nonincreasing rearrangement (Definition \ref{dfreprearr}) of the function 
 $X_\mu(\e,\bs)$ (defined by \eqref{dfWbsbt}) with respect to measures $\mu$ from the set $BP_{eq}(\mathcal{G}_r(\by)\subseteq BP(\mathcal{G}_r(\by)$, consisting of measures which are equivalent to the Lebesgue measure.
 
 Corollary \ref{crdoublrear} is the immediate consequence of Theorem \ref{thdoublrear} for the case where the domains $\mathcal{G}_r(\by)$ are balls $B_r(\by)$. In the formulation of it the set $BP_{eq}(B_r(\by))$ is replaced by its part $M_f(\by,r)$ (Definition \ref{dfMacry}).
 
 Theorem \ref{thdoubrearrleb}, based on Corollary \ref{crdoublrear}. yields a sufficient condition for discreteness of the spectrum of $H$ in terms of repeated nonincreasing rearrangement with respect to Lebesgue measure of the function
 \begin{equation}\label{dfYmust}
 Y(\e,\bt)=\sqrt{V(\e)}\frac{1}{|\e-\bt|^{d-2}}\sqrt{V(\bt)}
 \end{equation} 
on cubes that are going to infinity.
 
 In Proposition \ref{prrellebesdoublebes} we compare Theorem \ref{thdoubrearrleb} and  Theorem \ref{thcondlebesgue}.
 
Theorem \ref{thlogmdensdoubrearr} is based on Theorem \ref{thdoubrearrleb}. It yields an easier verifiable  condition for discreteness of the spectrum of $H$  by use of $m$-adic partition of a unit cube and of our concept of  $(\log_m,\,\theta)$- dense system of subsets of this cube (Definition \ref{dfdenslogmtetpart1}).

The paper is organized as follows. After this Introduction, in Section \ref{sec:prel} (Preliminaries)  we introduce some concepts and notations used in the paper. In Section \ref{sec:prevwork} we formulate some results from the previous work \cite{Zel1}, used in this paper. In Section \ref{sec:mainres} we formulate the main results of the paper and in Section \ref{sec:proofmainres} we prove them. In Section \ref{sec:examples} we recall briefly some examples, constructed in \cite{Zel1}, and construct a counterexample (Example \ref{exdoubrearlebes}) which shows that Theorem \ref{thdoubrearrleb} is essentially more general than Theorem \ref{thcondlebesgue}.  Sections \ref{sec:A}, \ref{sec:B} and \ref{sec:C} are Appendices. In Section \ref{sec:A} we prove some claims concerning the existence of base polyhedron for harmonic capacity and its connection with Choquet integral by this capacity. In Section \ref{sec:B} we obtain upper and lower estimates for the composition of Bessel kernel with itself on a ball $B_r(\by)$. There we use the well known arguments usually applied to estimation of composition for singular radial kernels. But main difficulty was to obtain lower estimate \eqref{estXfrombelow} with a constant not depending on the radius $r$ of the ball $B_r(\by)$ while $r\in (0,r_0)$.  In Section \ref{sec:C} we prove existence of the repeated nonincreasind rearrangement for an integrable function $F(\e,\bs)$ defined on the product of two measure spaces. There we use some facts from the theory of Riesz spaces \cite{Lux-Za}.

\section{Preliminaries} \label{sec:prel}  
\setcounter{equation}{0}

Let us come to agreement on some notations and terminology. Let $\Omega$ be an open and bounded domain in $R^d$. We denote by
$\Sigma_B(\bar\Omega)$ the $\sigma$-algebra of all
Borel subsets of $\bar\Omega$. By $\Sigma_L(\bar\Omega)$ we denote the $\sigma$-algebra of all Lebesgue measurable subsets of $\bar\Omega$, i.e., it is the Lebesgue completion of $\Sigma_B(\bar\Omega)$ by the Lebesgue measure $\mathrm{mes}_d$. If $(X,\Sigma,\mu)$ is a measure space, we call all sets from $\Sigma$ $\,\mu$-{\it measurable} and if $X=\bar\Omega\subseteq\R^d$, $\mu=\mathrm{mes}_d$ and $\Sigma=\Sigma_L(\bar\Omega)$, we simply call them  measurable. If a measure is absolutely continuous with respect to Lebesgue measure, we simply call it absolutely continuous. By $B_r(\by)$ we denote the open ball in $\R^d$ whose radius and center are $r>0$ and $\by$. 

Let $\Sigma$ be a non-empty algebra of subsets of a set $X$, $B(\Sigma,X)$ be the set of bounded, real valued, $\Sigma$-measurable functions on $X$, and $v$ be a monotonic real valued function on $\Sigma$ with $v(\emptyset)=0$. Monotonicity means that for any $E$ and $F$ in $\Sigma$ $E\subseteq F$ implies $v(E)\le v(F)$. In \cite{Ch} Choquet defined the following integration operation with respect to the nonnecessarily additive set function $v$: for a nonnegative function $F\in B(\Sigma,X)$
\begin{equation}\label{dfgenChoqint}
\int_X F(x)\,v(dx):=\int_0^\infty v\big(\{x\in X:\, F(x)\ge t\}\big)\,dt.
\end{equation}

Let us recall the definition of the {\it harmonic (or Newtonian) capacity}\footnote[1]{In the Russian literature it is often called  \it{Wiener capacity.}}  of a compact set
$E\subset\R^d$ (\cite{M-Sh}):
\begin{eqnarray}\label{dfWincap}
&&\mathrm{cap}(E):=\inf\Big(\big\{\int_{\R^d}|\nabla u(\e)|^2\,
\mathrm{d}\e\,:\;u\in C^\infty(\R^d), \;u\ge
1\;\mathrm{on}\;E,\nonumber\\
&&u(\e)\rightarrow 0\;\mathrm{as}\;|\e|\rightarrow\infty \big\}\Big).
\end{eqnarray}
It is known \cite{Ch} that the set function ``cap'' can be extended in a suitable manner from the set of all compact subsets of the space $\R^d$ to the set of all Borel subsets of it.
It is known  (\cite{Maz}, \cite{Maz1}) that the set function ``cap''
is monotonic and {\it submodular} (concave) in the sense that for any pair of sets $A,\,B\in\Sigma_B(\bar\Omega)$
\begin{equation}\label{submod}
\mathrm{cap}(A\cup B)+\mathrm{cap}(A\cap
B)\le\mathrm{cap}(A)+\mathrm{cap}(B).
\end{equation}

In \cite{AH} a more general concept of capacity has been considered (\cite{AH}, p.25, Definitions 2.3.1, 2.3.3).  Let $g(\e)$ be a {\it radially decreasing
	convolution kernel}  (\cite{AH}, p. 38). Then the capacity $C_{g,p}$,corresponding to $g$ and $p>1$, is defined in the following manner:
\begin{equation}\label{dfCgp}
C_{g,p}(E)=\inf\Big\{\int_{\R^d}|h(\e)|^p\,
\mathrm{d}\e\,:\,h\in\Omega_E\Big\}\quad (E\subset \R^d),
\end{equation}
where $E\in\Sigma_B(\R^d)$ and
\begin{equation}\label{dfOmE}
\Omega_E=\{h\in L_p(\R^d)\,:\,(g\star h)(\e)\ge
1\;\;\mathrm{for}\;\mathrm{all}\;\;\e\in E\}
\end{equation}

In \cite{AH} (p.189) the following {\it capacitary strong type
	inequality} has been established:
\begin{equation}\label{strtypineq}
\int_{\R^d}\big((g\star f)(\e)\big)^p\,C_{g,p}(\mathrm{d}\,\e)\le
A\int_{\R^d}\big(f(\e)\big)^p\, \mathrm{d}\e,
\end{equation}
where $f$ is a non-negative function belonging to
$L_p(\R^d)\;(1<p<\infty))$ and the constant $A>0$ does not
depend on $f$.  The  integral in left hand side of \eqref{strtypineq} is the Choquet
integral \eqref{dfgenChoqint} by the non-additive set function $v=C_{g,p}$ (here $X=\R^d$ and $F(\e)=\big((g\star f)(\e)\big)^p$).

By $M(\bar\Omega)$ denote the set of all  additive set functions on
$\Sigma_B(\bar\Omega))$ (we shall call them briefly ``measures'') and by $M^+(\bar\Omega)$ denote the set of
all non-negative measures from $M(\bar\Omega))$. In the theory of
coalition games (\cite{Shap} \cite{Schm}, \cite{Mar-Mon}) the concept of the
{\it core} of a game is used. 
Following to
\cite{Fuj}, we define for the harmonic capacity on $\bar\Omega$ a dual concept of the {\it base
	polyhedron} $\mathrm{BP}(\bar\Omega)$: 
\begin{eqnarray}\label{defcore}
&&\mathrm{BP}(\bar\Omega):=\{\mu\in M^+(\bar\Omega)):\;\\
&&\mu(A)\le \mathrm{cap}(A)\;\mathrm{for\; all}\;
A\in\Sigma_B(\bar\Omega))\;\mathrm{and}\;\mu(\bar\Omega)=\mathrm{cap}(\bar\Omega)\}.\nonumber
\end{eqnarray}
This set is nonempty, convex, and compact in the weak*-topology
(Proposition \ref{prcore}). If $\Omega=\mathcal{G}_r(\by)$ (see \eqref{dfcalGry}. \eqref{formofcalG}), we shall
write briefly $M^+(\by,r)$ and $\mathrm{BP}(\by,r)$. Denote by
$\mathrm{BP}_{eq}(\bar\Omega)$ the subset of 
$\mathrm{BP}(\bar\Omega)$ 
consisting of Radon measures which are equivalent to
Lebesgue measure $\mathrm{mes}_d$ (\cite{Hal}).

Let us recall the concept
of {\it measure preserving mapping} (\cite{Car-Dan}, Definition
2). 
\begin{definition}\label{dfmeaspres}
	Let $(\Omega,\,\mathcal{A},\,Q)$ be a  probability space
	and $\lambda$ be Lebesgue measure on $[0,1]$. A measurable
	function $s:\,\Omega\rightarrow [0,1]$ is called {\it measure
		preserving} , if $\lambda(B)=Q\big(s^{-1}(B)\big)$ for any Borel
	subset of $[0,1]$. Denote by $\mathcal{S}(\Omega,\,Q)$ the
	collection of all such functions.	
\end{definition}

Suppose that $\Omega=\bar
B_r(\by)$, $\mathcal{A}=\Sigma_L(\bar B_r(\by))$ 
and $Q$ is the normalized
Lebesgue measure $m_{d,r}$ on $\bar B_r(\by)$, defined by:
\begin{equation}\label{dfmdr}
m_{d,r}(A):=\frac{\mathrm{mes}_d(A)}{\mathrm{mes}_d(B_r(\by))}\quad(A\in\Sigma_L(\bar B_r(\by))).
\end{equation}
In \cite{Zel1} we have used the following set of measures, which is a part of $\mathrm{BP}_{eq}(\bar
B_r(\by))$ (Proposition \ref{lmdescribedens}):
\begin{definition}\label{dfMacry}
	Consider the function $f(t)=t^{(d-2)/d}\;(t\in[0,1])$ and denote by ${\mathrm M}_f(\by,r)$ the set of absolute
	continuous measures on $\bar B_r(\by)$, whose densities run over
	the following convex set:
	\begin{equation}\label{dfcalCo}
	{\mathcal Co}\,(\by,\,r):=\mathrm{cap}(\bar
	B_r(0))\cdot\,\overline{\mathrm{co}}\Big(\{f^\prime\,\circ\,s:\,s\in\mathcal{S}(\bar
	B_r(\by),\,m_{d,r})\}\Big),
	\end{equation}
	where  $``\mathrm{co}''$ denotes the convex
	hull and the closure is taken for the $L_1(\bar B_r(\by),\\m_{d,r})$
	topology.
\end{definition}
\begin{remark}\label{remoncubes}
	We shall consider the unit cube $Q=[-1,\,1]^d\subset\R^d$ and the dilation of it $Q_r:=r\cdot Q\,(r>0)$ and the translation of the latter 
	$Q_r(\by):=Q_r+\by$.	
\end{remark}
Denote by $[x]$ the integer part of a real number $x$. In \cite{Zel1} we have used the following concepts:
\begin{definition}\label{dfregpar1}
	We call a subset of $\R^d$ a {\it regular parallelepiped}, if it has the form $\times_{k=1}^d[a_k,\,b_k]$.
\end{definition}
\begin{definition}\label{dfdenslogmtetpart1}
	Suppose that $m>1$	and $\theta\in(0,1)$. A sequence $\{D_n\}_{n=1}^\infty$ of subsets of a cube $Q_1(\by)$ is said to be a {\it $(\log_m,\,\theta)$- dense system} in $Q_1(\by)$, if 
	
	(a) each $D_n$ is a finite union of regular parallelepipeds;
	
	(b) for any 
	cube $Q_r(\bz)\subseteq Q_1(\by)$ with $r\in\big(0,\,\min\{1,\,\frac{1}{\theta m^2}\}\big)$ there is  
	$j\in\{1,2,\dots, \big[\log_m\big(\frac{1}{\theta r}\big)\big]\}$
	such that for some regular parallelepiped $\Pi\subseteq D_j$ 
	there is a cube $Q_{\theta r}(\bs)$, contained in $\Pi\cap Q_r(\bz)$.
\end{definition}

\section{Some results from the previous work} \label{sec:prevwork}
\setcounter{equation}{0}

All the results of our previous paper \cite{Zel1} and of the present one are based on the following optimization problem for set functions:

\begin{problem}\label{extrprobl}
	Let $(X,\,\Sigma,\,\mu)$ be a measure space with a non-negative measure
	$\mu$  and $W(x)$ be a
	non-negative function defined on $X$ and belonging to
	$L_1(X,\,\mu)$. For $t\in(0,\,\mu(X))$ consider the
	collection $\mathcal{E}(t,X,\,\mu)$ of all $\mu$-measurable
	sets $E\subseteq X$ such that $\mu(E)\ge t$. The goal is to find
	the quantity
	\begin{equation}\label{dfIVOm}
	I_W(t,\,X,\,\mu)=\inf_{E\in\mathcal{E}(t,X,\,\mu)}\int_E W(x)\,
	\mu(\mathrm{d}x).
	\end{equation}
\end{problem}

In the formulation of next claim, proved in \cite{Zel1}, we have  used the following notations.
For the measure space and the function $W(x)$, introduced in Problem \ref{extrprobl},
consider the quantity:
\begin{equation}\label{dfJVOm}
J_W(t,\,X,\mu):=\int_{\mathcal{K}_W^-\big(t,\,X,\,\mu\big)}W(x)\,
\mu(\mathrm{d}x)+\big(t-\kappa_W^-(t,\,X,\,\mu)\big) W_\star(t,\,X,\,\mu),
\end{equation}
where $W_\star(t,\,X,\,\mu)$ is the non-decreasing rearrangement of the function $W(x)$, i.e,,
\begin{equation}\label{dfstOm}
W_\star(t,\,X,\,\mu):=\sup\{s>0:\;\lambda_\star(s,\,W,\,X,\,\mu)<
t\}\quad (t>0)
\end{equation}
with
\begin{equation}\label{dfFVOm}
\lambda_\star(s,\,W,\,X,\,\mu)=\mu({\mathcal L} _\star(s,W,X)),
\end{equation}
\begin{equation}\label{dfKVOm}
{\mathcal L} _\star(s,W,X)=\{x\in X:\;W(x)\le s\}.
\end{equation}
Furthermore,
\begin{equation}\label{dfKWt}
\mathcal{K}_W^-\big(t,\,X,\,\mu\big)={\mathcal L} _\star(s^-,W,X)\vert_{s=W_\star(t,\,X,\,\mu)},
\end{equation}
and
\begin{equation}\label{dfFVommin}
\kappa_W^-(t,\,X,\,\mu)=\mu\big(\mathcal{K}_W(t,\,X,\,\mu)\big),
\end{equation}
where
\begin{equation}\label{dfKVsminOm}
{\mathcal L} _\star(s^-,W,X)=\bigcup_{u<s}{\mathcal L} _\star(u,W,X)=\{x\in
X:\;W(x)< s\},
\end{equation}

The following claim from \cite{Zel1} solves Problem \ref{extrprobl} for a non-atomic  measure:

\begin{proposition}\label{prsolextrprob}[\cite{Zel1}, Theorem 3.3]
	Suppose that, in addition to  conditions of Problem \ref{extrprobl}, the measure 
	$\mu$ is non-atomic. Then
	
	(i) for any $t\in(0,\,\mu(X))$ there exists a $\mu$-measurable set $\tilde{\mathcal K}\subseteq X$ such that 
	\begin{equation}\label{estmesKV}
	\mu(\tilde{\mathcal K})=t,
	\end{equation}
	for the quantity $J_W(t,\,X,\,\mu)$, defined by
	\eqref{dfJVOm}-\eqref{dfFVommin}, the representation
	\begin{equation}\label{reprJV}
	J_W\big(t,\,X,\,\mu\big)= \int_{\tilde{\mathcal
			K}}W(x)\, \mu(\mathrm{d}x)
	\end{equation}
	is valid and
	\begin{eqnarray}\label{propKV}
	&&\forall\;x\in\tilde{\mathcal K}:\quad W(x)\le
	W_\star(t,\,X,\,\mu)\nonumber,\\
	&&\forall\;x\in X\setminus\tilde{\mathcal K}:\quad W(x)\ge W_\star(t,\,X,\,\mu);
	\end{eqnarray}
	
	(ii) the equality
	\begin{equation}\label{IeqJ}
	I_W(t,\,X,\,\mu)=J_W(t,\,X,\,\mu)
	\end{equation}
	is valid.
\end{proposition}

In the next claim. proved in \cite{Zel1}, we have obtained a two-sided estimate for the solution $J_W(t,\,X,\,\mu)$ of Problem \ref{extrprobl}  via a non-increasing rearrangement of the function $W(x)$ on  $X$. 
This rearrangement is following:
\begin{equation}\label{dfSVyrdel}
 W^\star(t;\,X;\,\mu):=\sup\{s>0\,:\;\lambda^\star(s;\,\,W;\,X;\,\mu)\ge
t\}\quad (t>0)
\end{equation}
where
\begin{eqnarray}\label{dfLVsry}
&&\lambda^\star(s;\,W;\,X;\,\mu)=\mu(\mathcal{L}^\star(s;\,W;\,X;\,\mu)),\nonumber\\
&&\mathcal{L}^\star(s;\,W;\,X)=\{x\in X,:\; W(x)\ge s\}.
\end{eqnarray}

The promised claim is following:

\begin{proposition}\label{lmJR}[\cite{Zel1}, Proposition 3.4]
	Suppose that, in addition to conditions of Problem \ref{extrprobl} and Theorem
	\ref{prsolextrprob}, the measure $\mu$ is finite.
	Then for $\theta>1$	and $t\in(0,\,\mu(X))$ the estimates
	\begin{equation}\label{estJR1}
	J_W(\mu(X)-t/\theta,X,\mu))\ge\frac{(\theta-1)t}{\theta} W^\star(t;\,X;\,\mu),
	\end{equation}
	\begin{equation}\label{estJR2}
	J_W(\mu(X)-t,\,X,\mu)\le (\mu(X)-t) W^\star(t;\,X;\,\mu)
	\end{equation}
	are valid.
\end{proposition}

Let us formulate some claims from \cite{Zel1}, which were obtained with the help of arguments, mentioned above. 

\begin{proposition}\label{lmmeasinstcap1}[\cite{Zel1}, Theorem 3.7]
	Suppose that for some $r_0>0$ and any $r\in(0,r_0)$ the
	condition 
	\begin{equation}\label{cndSVy}
	\lim_{|\by|\rightarrow\infty}\ V^\star(\delta(r);\,\by,r)=\infty
	\end{equation}
	is satisfied with $\delta(r)=\hat\gamma(r)\mathrm{mes}_d(\mathcal{G}_r(0))$ and
	$\gamma(r)$ satisfies the conditions 
		\begin{equation}\label{cndtildgam}
	\forall\,r\in(0,\,r_0):\;\tilde\gamma(r)\in(0,1)\quad\mathrm{and}\quad\limsup_{r\downarrow
		0}\,r^{-2(d-2)/d}\,\tilde\gamma(r)=\infty.
	\end{equation}
	Then the spectrum of the operator $H=-\Delta+V(\e)$ is discrete.
\end{proposition} 

The following claim describes a part of $\mathrm{BP}_{eq}(\by,\,r)$:

\begin{proposition}\label{lmdescribedens}[\cite{Zel1}, Proposition A.2]
	The set $\mathrm{BP}_{eq}(\by,\,r)$ contains the set ${\mathrm
		M}_f(r,\by)$ of absolute continuous measures described in
	Definition \ref{dfMacry}.
\end{proposition}

 In \cite{Zel1} we have
denoted by $\alpha_\mu(\e)$ $(\mu\in\mathrm{BP}_{eq}(\by,r))$ the density of the measure
$\mathrm{mes}_d$ with respect to $\mu$, i.e.,
\begin{equation}\label{dfalphamu}
\alpha_\mu:=\frac{\mathrm{d}\,\mathrm{mes}_d}{\mathrm{d}\,\mu}.
\end{equation}

In the following claim the set $M_f(\by,r)$ has been used: 

\begin{proposition}\label{crusecoredstort}[\cite{Zel1}, Corollary 3.12]
	Suppose that $d\ge 3$, $V(\e)\ge 0$, $V\in L_{1,\,loc}(\R^d)$.
	If the condition
	\begin{equation}\label{cndZmunew}
	\lim_{|\by|\rightarrow\infty}\;\sup_{\mu\in\mathrm{M}_f(\by,\,r)}Z_{\mu}^\star(\psi_\mu(r);\,
	B_r(\by);\,\mu)=\infty
	\end{equation}
	is satisfied for some $r_0>0$ and any $r\in(0,r_0)$, where $Z_{\mu}(\e)=\alpha_\mu(\e)\, V(\e)$,
	$\psi_\mu(\by,r)=\gamma(r)\mu(B_r(\by))$ and $\gamma(r)$ satisfies
	conditions \eqref{cndgammar}.
	Then the spectrum of the operator
	$H=-\Delta+V(\e)$ is discrete.
\end{proposition}

Consider the covering of the space $\R^d$ by the cubes $Q_1(\vec l)\;(\vec l\in\Z^d)$ and for any $\vec l\in\Z^d$ consider a sequence $\{D_j(\vec l)\}_{j=1}^\infty$ of 
subsets of $Q_1(\vec l)$. Furthermore, for some integers $n>0$ and $m>1$ consider the $m$-adic partition of each cube $Q_1(\vec l)$: $\{Q(\vec\xi,n)\}_{\vec\xi\in\Xi_n(\vec l)}$,
where $Q(\vec\xi,n)=Q_{m^{-n}}(\vec\xi)$ and $\Xi_n(\vec l)=\{\vec\xi\in m^{-n}\cdot\Z^d\,:\,Q(\vec\xi,n)\subset  Q_1(\vec l)\}$. Denote
\begin{equation}\label{dfXijvecl}
\Xi_n(\vec l,\,j)=\{\vec\xi\in m^{-n}\cdot\Z^d\,:\,Q(\vec\xi,n)\subseteq  D_j(\vec l)\},
\end{equation} 

The following claim was based on the previous claim and on the concept of a $(\log_m,\,\theta)$- dense system (Definition \ref{dfdenslogmtetpart1}):

\begin{proposition}\label{prlogmdensdoubrearr}[\cite{Zel1}, Theorem 3.14]
	Suppose that $\theta\in(0,1)$ and for each $\vec l\in\Z^d$ the sequence $\{D_j(\vec l)\}_{j=1}^\infty$ forms a $(\log_m,\,\theta)$- dense system in $Q_1(\vec l)$. Furthermore,
	suppose that 
	\begin{equation}\label{Xinonempty2}
	\forall\;\vec l\in\Z^d,\;n\in\N,\;j\in\{1,2.\dots,n\}:\quad \Xi_{n}(\vec l,j)\neq\emptyset. 
	\end{equation}
	Let $\gamma(r)$ be a nondecreasing monotone function 
	satisfying condition \eqref{cndgammar}. If for 
	any natural $n$ the condition 
	\begin{equation}\label{cndlogdensdoubrearrprev}
	\lim_{|\vec l|\rightarrow\infty}\;\min_{\vec\xi\in\bigcup_{j=1}^n\Xi_{n}(\vec l,\,j)} V^\star\big(\psi(m,n);\,Q(\vec\xi,n)\big)=\infty
	\end{equation}
	is satisfied with $\psi(m,n)=\gamma(m^{-n})\mathrm{mes}_d(Q(\vec 0,n))$, then the spectrum of the operator $H=-\Delta+V(\e)\cdot$ is discrete.
\end{proposition}

\section{Main results} \label{sec:mainres}
\setcounter{equation}{0} 

On the base of Proposition \ref{crusecoredstort} we shall prove the following theorem:

\begin{theorem}\label{thcondlebesgue}
	If for some $r_0>0$ and a  function $\gamma(r)$,	satisfying conditions \eqref{cndgammar}, the condition
	\begin{equation}\label{cndlebesgue}
	\lim_{|\by|\rightarrow\infty} V^\star(\gamma(r)\mathrm{mes}_d(Q_r(\by));\,Q_r(\by))=\infty
	\end{equation}
	is satisfied for any $r\in(0,r_0]$, then the spectrum of the operator
	$H=-\Delta+V(\e)$ is discrete.
\end{theorem} 

The following relation between Theorem \ref{thcondlebesgue} and Proposition \ref{prlogmdensdoubrearr} is valid:
\begin{proposition}\label{threlprlogdenstlebes}
	If  the conditions of Proposition \ref{prlogmdensdoubrearr} are satisfied, then  the conditions of Theorem \ref{thcondlebesgue} are fulfilled.  
\end{proposition}

\begin{remark}
	Since conditions \eqref{cndtildgam} are stronger than conditions \eqref{cndgammar}, Proposition \ref{lmmeasinstcap1} (with $Q_r(\by)$ instead of $\mathcal{G}_r(\by$)) implies Theorem \ref{thcondlebesgue}. On the other hand, in \cite{Zel1} (Examples 5.4, 5.5) an example of the potential $V(\e)$ has been constructed, which satisfies conditions of 
	Proposition \ref{prlogmdensdoubrearr}, but conditions of Proposition \ref{lmmeasinstcap1} are not fulfilled for it. In view of Proposition \ref{threlprlogdenstlebes}, this means that Theorem \ref{thcondlebesgue} is essentially more general than Proposition \ref{lmmeasinstcap1}. 
\end{remark}

If $\Omega=\mathcal{G}_r(\by)$ (see \eqref{dfcalGry}. \eqref{formofcalG}), we shall
write briefly $M(\by,r)$ and $\mathrm{BP}(\by,r)$. For $\mu\in
M(\by,r)$ denote by
$\mathcal{M}_\gamma^\mu(\by,r)\;(\gamma\in(0,1))$ the collection
of all compact sets $F\subseteq\bar{\mathcal G}_r(\by)$ satisfying the
condition
\begin{equation}\label{defMmu} 
\mu(F)\le\gamma\,\mu(\mathcal{G}_r(\by)).
\end{equation}

In all next claims we suppose that $V\in L_{p,\,loc}(\R^d)$ with
$p>d/2$.
Denote by $\mathcal{F}$ the Fourier transform on $\R^d$ and
consider the {\it Bessel kernel} of the order $1$:
\begin{equation}\label{dfBesskern}
G_1(\e)=\mathcal{F}^{-1}\big((1+|\xi|^2)^{-1/2}\big)(\e)
\end{equation}
(\cite{AH}, p. 10).
\begin{theorem}\label{thusecore}
	Suppose that for some $r_0>0$ and any $r\in(0,r_0)$
	\begin{eqnarray}\label{cndusecore}
	&&\hskip-5mm\lim_{|\by|\rightarrow\infty}\;\sup_{\mu\,\in\,
		\mathrm{BP}(\by,r)}\;\inf_{F\in\mathcal{M}^\mu_{\gamma(r)}(\by,r)}\int_{\bar{\mathcal G}_r(\by)}\Big(\int_{\mathcal{G}_r(\by)\setminus
		F}G_1(\e-\bs)\sqrt{V(\bs)}\,\mathrm{d}\bs\Big)^2\mu( \mathrm{d}\e)\nonumber\\
	&&\hskip-5mm=\infty,
	\end{eqnarray}
	where $\gamma(r)$ satisfies the condition \eqref{cndgammar}.
	Then the spectrum of the
	operator $H=-\Delta+V(\e)$ is discrete.
\end{theorem}


Before formulation of the following theorem, based on Theorem \ref{thusecore}, let us introduce a  
some notations. 

Let $(X,\,\Sigma,\,\mu)$, $(Y,\,\Xi,\,\nu)$ be two measure spaces with  non-negative measures 
$\mu$, $\nu$ and $F:\,X\times Y\rightarrow\R$ be a nonnegative $\mu\times\nu$-measurable function, belonging to $L_1(X\times Y,\,\mu\times\nu)$.  
In Section \ref{sec:C} we have defined the repeated nonincreasing rearrangement $(F^\star)^\star(t,u;\,X,Y;\,\mu,\nu)\,(t>0,u>0)$ of the function $F(\e,\bs)$ and proved its existence (Definition \ref{dfreprearr} and Proposition \ref{prexistpartrear}).
If $X=Y$, $\Sigma=\Xi$ and $\mu=\nu$, we shall write $(F^\star)^\star(t,u;\,X;\,\mu)$ and if in addition $t=u$, we shall write $(F^\star)^\star(t;\,X;\,\mu)$. If in the last case $X=\mathcal{G}_r(\by)$, we shall write $F^\star(\cdot,\bs)(t;\, \by,r;\,\nu)$, $(F^\star)^\star(,t;\,\by,r;\,\mu)$ and $(F^\star)^\star(t;\,\by,r;\,\mu)$ and if in addition
$\mu=\mathrm{mes}_d$, we shall omit $\mathrm{mes}_d$ in the above notations.	

Denote for $\mu\,\in\,\mathrm{BP}_{eq}(\by,r)$:
\begin{equation}\label{dfWbsbt}
X_\mu(\bs,\bt)=\sqrt{V(\bs}\sqrt{V(\bt)}\alpha_\mu(\bs)\alpha_\mu(\bt)K_\mu(\bs,\bt),
\end{equation}
where
\begin{equation}\label{dfKmubsbt}
K_\mu(\bs,\bt)=\int_{\bar{\mathcal G}_r(\by)}G_1(\e-\bs)G_1(\e-\bt)\,\mu(\mathrm{d}\e).
\end{equation}

Since each measure $\mu\in\mathrm{BP}_{eq}(\by,r)$ is equivalent to the Lebesgue measure
$\mathrm{mes}_d$, the Lebesgue completion of $\Sigma_B(\mathcal{G}r(\by))$ by $\mu$ coincides 
with \\$\Sigma_L(\mathcal{G}r(\by))$. Hence we can consider the complete measure space
\begin{equation*}
\big(\mathcal{G}r(\by),\,\Sigma_L(\mathcal{G}r(\by)),\,\mu\big).
\end{equation*} 

\begin{theorem}\label{thdoublrear}
	Suppose that the condition
	\begin{eqnarray}\label{cnddoublrear}
	\hskip-5mm\lim_{|\by|\rightarrow\infty}\;\sup_{\mu\,\in\,
		\mathrm{BP}_{eq}(\by,r)}\;(X_\mu^\star)^\star(\psi(r);\,\by,r;\,\mu)=\infty
	\end{eqnarray}
	is fulfilled for some $r_0>0$ and  any $r\in(0,r_0]$, where $\psi_\mu(r)=\gamma(r))\mu(\mathcal{G}_r(0))$ and $\gamma(r)$ satisfies
	conditions \eqref{cndgammar}. Then the spectrum of the operator
	$H=-\Delta+V(\e)$ is discrete.
\end{theorem}

The following consequence of the previous theorem is valid:

\begin{corollary}\label{crdoublrear}
	If in the formulation of Theorem \ref{thdoublrear} $\;\mathcal{G}_r(\by)=B_r(\by)$ and instead of condition \eqref{cnddoublrear} the condition
	\begin{eqnarray}\label{cnddoublrear1}
	\hskip-5mm\lim_{|\by|\rightarrow\infty}\;\sup_{\mu\,\in\,
		M_f(\by,r)}\;(X_\mu^\star)^\star(\psi(r);\,\by,r;\,\mu)=\infty
	\end{eqnarray}
	is satisfied,  then the spectrum of the operator
	$H=-\Delta+V(\e)$ is discrete.
\end{corollary}

Recall that the function $Y(\e,\bt)$ is defined by \eqref{dfYmust}.
The following theorem is based on the previous claim:

\begin{theorem}\label{thdoubrearrleb}
	Suppose that the condition
\begin{equation}\label{cnddoubrearrleb}
\lim_{|\by|\rightarrow\infty} (Y^\star)^\star\big(\psi(r);\,Q_r(\by)\big)==\infty
\end{equation}
is fulfilled for some $r_0>0$ and  any $r\in(0,r_0]$, where
\begin{equation*}
 \psi_\mu(r)=\gamma(r))\mathrm{mes}_d(Q_r(\by)) .
\end{equation*}
and $\gamma(r)$ satisfies
conditions \eqref{cndgammar}. Then the spectrum of the operator
$H=-\Delta+V(\e)$ is discrete.	
\end{theorem}

The following relation between the previous claim and Theorem \ref{thcondlebesgue} is valid:

\begin{proposition}\label{prrellebesdoublebes}
	If the conditions of Theorem \ref{thcondlebesgue} are satisfied, 
	then the conditions of Theorem \ref{thdoubrearrleb} are fulfilled.
\end{proposition}

The following theorem, based on Theorem \ref{thdoubrearrleb}, uses the  concept of a $(\log_m,\,\theta)$- dense system (Definition \ref{dfdenslogmtetpart1}) like Proposition \ref{prlogmdensdoubrearr}. Recall that the sets $\Xi_n(\vec l,\,j)$ were defined by 
\eqref{dfXijvecl}. 

\begin{theorem}\label{thlogmdensdoubrearr}
	Suppose that $\theta\in(0,1)$ and for each $\vec l\in\Z^d$ the sequence $\{D_j(\vec l)\}_{j=1}^\infty$ forms a $(\log_m,\,\theta)$- dense system in $Q_1(\vec l)$. Furthermore,
	suppose that condition \eqref{Xinonempty2} is satisfied,
	Let $\gamma(r)$ be a nondecreasing monotone function 
	satisfying condition \eqref{cndgammar}. If for 
	any natural $n$ the condition 
	\begin{eqnarray}\label{cndlogdensdoubrearr}
	\hskip-8mm&&\lim_{|\vec l|\rightarrow\infty}\;\min_{\vec\xi,\,\vec\eta\in\bigcup_{j=1}^n\Xi_{n}(\vec l,\,j),\;|\vec\xi-\vec\eta|\le\frac{\sqrt{d}}{\theta}m^{-(n-2)}} (Y^\star)^\star\big(\psi(m,n);\,Q(\vec\xi,n),\,Q(\vec\eta,n)\big)=\nonumber\\
	\hskip-8mm&&\infty
	\end{eqnarray}
	is satisfied with $\psi(m,n)=\gamma(m^{-n})\mathrm{mes}_d(Q(\vec 0,n))$, then the spectrum of the operator $H=-\Delta+V(\e)\cdot$ is discrete.
\end{theorem}

\section{Proof of main results} \label{sec:proofmainres}
\setcounter{equation}{0}

\subsection{Proof of Theorem \ref{thcondlebesgue}}
\begin{proof}
	Let us take a ball $B_r(\by)$ and consider on it the probability measure $m_{d,r}$, defined by \eqref{dfmdr}. 
	Consider the function $f(t)=t^{(d-2)/d}$ and the absolutely continuous measure $\mu_s$ on $B_r(\by)$, whose density is $f^\prime\circ s_{r,\by}$. This means that for any measurable set $A\subseteq B_r(\by)$
	the equality  is valid
	\begin{equation}\label{dfmusnewest}
	\mu_s(A)=\mathrm{cap}(B_r(0))\int_Af^\prime(s_{r,\by}(\e))m_{d,r}(\mathrm{d}\e),
	\end{equation}
	where the function $s_{r,\by}:\,B_r(\by)\rightarrow [0,1]$ has the form:
	\begin{equation}\label{dfsrbynewest}
	s_{r,\by}(\e):=m_{d,r}\big(\{\bs\in B_r(\by):\,P_1\bs\le P_1\e\}\big)
	\end{equation}
	and $P_1$ is the
	following operator $P_1:\,\R^d\rightarrow\R$:
	\begin{equation}\label{dfP1newest}
	\mathrm{for}\quad
	\e=(x_1,x_2,\dots,x_d)\quad P_1\e:=x_1.
	\end{equation}
	It is easy to check, that $s_{r,\by}$ is a measure preserving mapping. Hence the measure $\mu_s$ belongs to $M_f(r,\by)$.
	Using Lemma \ref{lmineqforrearnewest} with $t=\gamma(\tilde r)\;(\tilde r=r/\sqrt{d})$
	and 
	taking $W(\e)=V(\e)$, we get that for some $\kappa,\delta\in(0,1)$ and any $\by\in\R^d$, $r\in(0,1)$ there 
	are $\vec l(\by,r)\in\Z^d$ and a cube $Q_{\tilde r}(\tilde\by)\subseteq B_r(\by)\cap Q_1(\vec l(\by,r))$ 
	such that 
	\begin{eqnarray}\label{ineqforrearV1}
	&& Z^\star_{\mu_s}\big(\bar\gamma(r)\cdot\mu_s(B_r(\by));\,B_r(\by);\,\mu_s\big)\ge\nonumber\\
	&&\delta\cdot\ V^\star\big(\gamma(\tilde r)
	\mathrm{mes}_d(Q_{\tilde r}(\tilde\by));\,Q_{\tilde r}(\tilde\by)\big)
	\end{eqnarray}
	where $\bar\gamma(r)=\kappa\gamma(r/\sqrt{d})$. Notice that since the function $\gamma(r)$ satisfies conditions \eqref{cndgammar}, the function $\bar\gamma(r)$ satisfies this condition too for some $r_0>0$. Then condition \eqref{cndlebesgue} and Proposition \ref{crusecoredstort} imply the desired claim.
\end{proof}

In the proof of Theorem \ref{thcondlebesgue} we have used the following claim from \cite{Zel1}:

\begin{lemma}\label{lmineqforrearnewest}[\cite{Zel1}, Lemma 4.2]
	For some $\kappa,\delta\in(0,1)$ and any ball $B_r(\by)$ with $r\in(0,1)$ there are $\vec l=\vec l(\by,r)\in\Z^d$ and a cube 
	\begin{equation}\label{incltointersectnew}
	Q_{\tilde r}(\tilde\by)\subseteq B_r(\by)\cap Q_1(\vec l(\by,r))
	\end{equation}
	with $\tilde r=r/\sqrt{d}$ such that for any nonnegative function 
	$W\in L_1(B_r(\by))$ and $t\in(0,1)$ the inequality
	\begin{equation}\label{ineqforrearnew}
	 Z^\star_{\mu_s,\,W}\big(\kappa t\cdot\mu_s(B_r(\by));\,B_r(\by);\,\mu_s\big)\ge\delta\cdot W^\star\big(t\cdot\mathrm{mes}_d(Q_{\tilde r}(\tilde\by));\,Q_{\tilde r}(\tilde\by)\big)
	\end{equation}
	is valid, where $\mu_s$ is the measure on $B_r(\by)$, defined by \eqref{dfmusnewest} with $f(t)=t^{(d-2)/d}$ and $Z_{\mu_s,\,W}(\e)=W(\e)\big(f^\prime(s_{r,\by}(\e))\big)^{-1}$
	with the function $s_{r,\by}(\e)$, defined by \eqref{dfsrbynewest}, \eqref{dfP1newest}.
	\end{lemma}

\subsection{Proof of Proposition \ref{threlprlogdenstlebes}}

\begin{proof}
	Suppose that all the conditions of Proposition \ref{prlogmdensdoubrearr} are satisfied.
	Let us take a cube $Q_r(\by)$.  It is clear that there 
are $\vec l(\by,r)\in\Z^d$ and a cube $Q_{r/2}(\tilde\by)\subseteq Q_r(\by)\cap Q_1(\vec l(\by,r))$. In view of Lemma \ref{prestbarVstarbelownewest} with $W(\e) =V(\e)$, Lemma 
 \ref{lminclrear2} and condition \eqref{Xinonempty2}, there are $j\in\{1,2,\dots,n\}$ with 
$n=\big[\log_m\big(\frac{2}{\theta r}\big)\big]+2$
and $K\in(0,1)$, such that 
\begin{eqnarray}\label{ineqforrearV1newestrel}
\hskip-12mm&& V^\star\big(\bar\gamma(r)\cdot\mathrm{mes}_d(Q_r(0));\,Q_r(\by)\big)\ge\nonumber\\ 
\hskip-12mm&&V^\star\big(\hat\gamma(r/2,K)\cdot\mathrm{mes}_d(Q_{r/2}(0));\,Q_{r/2}(\tilde\by)\big)\ge
\nonumber\\
\hskip-12mm&&
\min_{\vec\xi:\,Q(\vec\xi,\,n)\subseteq F_j} V^\star\big(\psi(m,n);\,\,Q(\vec\xi,n)\big)\ge
\min_{\vec\xi\in\bigcup_{j=1}^n\Xi_{n}(\vec l,\,j)} V^\star\big(\psi(m,n);\,Q(\vec\xi,n)\big),
\end{eqnarray}
where
\begin{equation}\label{dfhatgamrKnewest}
\hat\gamma(\rho,K)=K\gamma(\theta \rho/m^2)\theta^d,
\end{equation} 
$\psi(m,n)=\gamma(m^{-n})\mathrm{mes}_d(Q(\vec 0,n))$, $\bar\gamma(r)=2^{-d}\hat\gamma(r/2,\,K)$ and $F_j$ is a non-empty union of cubes $Q(\vec\xi,\,n)$ with $\vec\xi\in m^{-n}\cdot\Z^d$, such that $F_j\subseteq Q_{r/2}(\tilde\by)\cap D_j(\vec l(\by,r))$. Notice that since the function $\gamma(r)$ satisfies conditions \eqref{cndgammar}, the function $\bar\gamma(r)$ satisfies this condition too for some $r_0>0$.
Since $Q_r(\by)\cap Q_1(\vec l(\by,r))\neq\emptyset$, estimate \eqref{ineqforrearV1newestrel} and 
condition \eqref{cndlogdensdoubrearrprev} imply that condition \eqref{cndlebesgue}
of Theorem \ref{thcondlebesgue} is satisfied with $\gamma(r)=\bar\gamma(r)$. Proposition \ref{threlprlogdenstlebes} is proven.	
\end{proof}

In the proof of Proposition \ref{threlprlogdenstlebes} f we have used the following claims, proved in \cite{Zel1}:

\begin{lemma}\label{prestbarVstarbelownewest}[\cite{Zel1}, Lemma 4.3]
	Suppose that in a cube $Q_1(\vec l)\,(\vec l\in\Z^d)$ there is is a sequence of  subsets $\{D_n(\vec l)\}_{n=1}^\infty$ forming in it a $(\log_m,\,\theta)$-dense system.
	Let $\gamma:\,(0,\,r_0)\rightarrow\R$ be a monotone nondecreasing function with $r_0=\min\{1,\,1/(m^2\theta)\}$. Then 
	for some  $K\in(0,1)$ and for
	any cube $Q_r(\by)\subset Q_1(\vec l)$ there are $j\in\{1,2,\dots,n\}$, with 
	$n=\big[\log_m\big(\frac{1}{\theta r}\big)\big]+2$,
	and a non-empty set $F_j\subseteq Q_r(\by)\cap D_j(\vec l)$, which is a union of cubes 
	$Q(\vec\xi\,,n)$ with $\vec\xi\in m^{-n}\cdot\Z^d$,
	such that for any nonnegative function $W\in L_1(Q_r(\by))$ the inequality
	\begin{eqnarray}
	&&\hskip-8mm	W^\star(\hat\gamma(r,K)\mathrm{mes}_d(Q_r(\by));\,\,Q_r(\by))\ge\nonumber\\
	&&\hskip-8mm	\min_{\vec\xi:\;Q(\vec\xi,\,n)\subseteq F_j} W^\star\big(\gamma(m^{-n})\mathrm{mes}_d(Q(\vec\xi,n));\,Q(\vec\xi,n)\big) 
	\end{eqnarray}
	is valid, where the function$\hat\gamma(r,K)$ is defined by \eqref{dfhatgamrKnewest}.
\end{lemma}

\begin{lemma}\label{lminclrear2}[\cite{Zel1}, Lemma 4.5]
	Let $\Omega_1$ and $\Omega_2$ be measurable subsets of $\R^d$ such that $\Omega_1\subseteq\Omega_2$ and $W(\e)$ be a non-negative measurable function 
	defined on $\Omega_2$. Then for any $t>0$ the inequality
	$ W^\star(t;\,\Omega_1)\le W^\star(t;\,\Omega_2)$
	is valid.
\end{lemma}

\subsection{Proof of Theorem \ref{thusecore}}

\begin{proof}

		Let us use the strong capacitary inequality \eqref{strtypineq}, taking
	$p=2$, , 
	$f(\e)=\sqrt{V(\e)}\chi_{\mathcal{G}_r(\by)\setminus F}(\e)$,
	$g(\e)=G_1(\e)$ and $F\in\mathcal{N}_{\gamma(r)}(\by,r)$ (
	definition \eqref{defNcap}). Here $\chi_E$ is the characteristic function of a set $E\subseteq \mathcal{G}_r(\by)$
	Then we have,
	using definition \eqref{dfgenChoqint} of the Choquet integral and monotonicity of the capacity $C_{g,p}$, defined by \eqref{dfCgp}, \eqref{dfOmE}:
	\begin{eqnarray*}
	&&\int_{\mathcal{G}_r(\by)\setminus F}V(\e)\, \mathrm{d}\e\ge\\
	&&A^{-1}\int_{R^d}\Big(\int_{\mathcal{G}_r(\by)\setminus
		F}G_1(\e-\bs)\sqrt{V(\bs)}\,\mathrm{d}\bs\Big)^2C_{G_1,2}(\mathrm{d}\,\e)\ge\\
	&&A^{-1}\int_{\bar{\mathcal G}_r(\by)}\Big(\int_{\mathcal{G}_r(\by)\setminus
		F}G_1(\e-\bs)\sqrt{V(\bs)}\,\mathrm{d}\bs\Big)^2
	C_{G_1,2}(\mathrm{d}\,\e).
	\end{eqnarray*}
	Since $V\in L_{p,\,loc}(\R^d)$ with $p>d/2$, 
	then by claim (ii) of Lemma \ref{lmC12eqWien}, the function in the
	brackets in the last integral is continuous and bounded in $\bar{\mathcal G}_r(\by)$, hence there it is a bounded Borel function. Furthermore,
	in view of claim (i) of Lemma \ref{lmC12eqWien}, 
	the inequality is valid for some constant $B>0$:
	\begin{eqnarray*}
	&&\int_{\mathcal{G}_r(\by)\setminus
		F}V(\e)\, \mathrm{d}\e\ge A^{-1}\times\\
	&&\int_0^\infty C_{G_1,2}\Big(\Big\{\e\in \mathcal{G}_r(\by):\;\Big(\int_{\mathcal{G}_r(\by)\setminus
		F}G_1(\e-\bs)\sqrt{V(\bs)}\,\mathrm{d}\bs\Big)^2\ge t\Big\}\Big)\,dt\ge\nonumber\\
	&&B\int_0^\infty \mathrm{cap}\Big(\Big\{\e\in \mathcal{G}_r(\by):\;\Big(\int_{\mathcal{G}_r(\by)\setminus
		F}G_1(\e-\bs)\sqrt{V(\bs)}\,\mathrm{d}\bs\Big)^2\ge t\Big\}\Big)\,dt=\nonumber\\
	&&B\int_{\bar{\mathcal G}_r(\by)}\Big(\int_{\mathcal{G}_r(\by)\setminus
		F}G_1(\e-\bs)\sqrt{V(\bs)}\,\mathrm{d}\bs\Big)^2\,\mathrm{cap}(\mathrm{d}\,\e).
	\nonumber\\
	\end{eqnarray*}
	Hence by Proposition \ref{prcore} and Lemma \ref{lminrchanginfsup},
	\begin{eqnarray}\label{estbelowintV}
	&&\inf_{F\in\mathcal{N}_{\gamma(r)}(\by,r)}\int_{\mathcal{G}_r(\by)\setminus
		F}V(\e)\, \mathrm{d}\e\ge  B\times\\
	&&\inf_{F\in\mathcal{N}_{\gamma(r)}(\by,r)}\;\sup_{\mu\,\in\,\mathrm{BP}(\by,r)}\int_{\bar{\mathcal G}_r(\by)}\Big(\int_{\mathcal{G}_r(\by)\setminus
		F}G_1(\e-\bs)\sqrt{V(\bs)}\,\mathrm{d}\bs\Big)^2\,\mu( \mathrm{d}\e)\ge\nonumber\\
	&&B\,\sup_{\mu\,\in\,\mathrm{BP}(\by,r)}\;\inf_{F\in\mathcal{N}_{\gamma(r)}(\by,r)}\int_{\bar{\mathcal G}_r(\by)}\Big(\int_{\mathcal{G}_r(\by)\setminus
		F}G_1(\e-\bs)\sqrt{V(\bs)}\,\mathrm{d}\bs\Big)^2\,\mu(
	\mathrm{d}\e).\nonumber
	\end{eqnarray}
	On the other hand,
	by definition \eqref{defcore} of the base polyhedron for the harmonic
	capacity on $\bar\Omega=\bar{\mathcal G}_r(\by)$, 
	$\mu\in\mathrm{BP}(\by,r)$ and any compact set $F\subseteq\bar
	{\mathcal G}_r(\by)$
	\begin{equation*}
	\frac{\mu(F)}{\mu(\bar
		{\mathcal G}_r(\by))}\le\frac{\mathrm{cap}(F)}{\mathrm{cap}(\bar{\mathcal G}_r(\by))}.
	\end{equation*}
	Hence, in view of definitions \eqref{defNcap} and \eqref{defMmu},
	the inclusion is valid: 
	\begin{equation*}
	\mathcal{N}_{\gamma(r)}(\by,r)\subseteq
	\mathcal{M}^\mu_{\gamma(r)}(\by,r).
	\end{equation*}
	 This circumstance and
	inequality \eqref{estbelowintV} imply that
	\begin{eqnarray*}
		&&\inf_{F\in\mathcal{N}_{\gamma(r)}(\by,r)}\int_{\mathcal{G}_r(\by)\setminus
			F}V(\e)\, \mathrm{d}\e\ge\\
		&&B\,\sup_{\mu\in\mathrm{\mathrm{BP}(\by,r)}}\;\inf_{F\in\mathcal{M}^\mu_{\gamma(r)}(\by,r)}\int_{\bar{\mathcal G}_r(\by)}\Big(\int_{\mathcal{G}_r(\by)\setminus
			F}G_1(\e-\bs)\sqrt{V(\bs)}\,\mathrm{d}\bs\Big)^2\,\mu(
		\mathrm{d}\e).
	\end{eqnarray*}
	This estimate and condition \eqref{cndusecore} imply that
	condition \eqref{cndmolch1} of  Theorem \ref{thMazSh}  is
	fulfilled. Hence the spectrum of the operator $H$ is discrete and
	non-negative. Theorem \ref{thusecore} is proven.
\end{proof}

Let us prove some claims used in the proof of Theorem
\ref{thusecore}.


\begin{lemma}\label{lmC12eqWien}
(i)	For the capacity $C_{G_1,2}$, defined by
	\eqref{dfCgp}-\eqref{dfOmE} with $p=2$ and $g(\e)=G_1(\e)$
	and the harmonic capacity $\mathrm{cap}$ there exists $\;C>0$ such
	that for any compact set $E\subset\R^d$ $C_{G_1,2}(E)\ge
	C\,\mathrm{cap}(E)$;
	
	(ii) If $g\in L_p(\Omega)$ with $p>d$, then the function $f(\e)=G_1\star g(\e)$ is continuous and bounded in $\R^d$.
\end{lemma}
\begin{proof}
(i)	By Proposition 2.3.13 from \cite{AH} (p.29), $C_{G_1,p}=C_{1,p}\;(p>1)$,
	where $C_{1,p}$ is defined in the following manner:
	\begin{equation*}
	C_{1,p}(E)=\inf\Big(\big\{\Vert
	u\Vert_{1,2p}^p\,:\;u\in\mathcal{S}, \;u\ge
	1\;\mathrm{on}\;E\big\}\Big).
	\end{equation*}
	Here $\mathcal{S}$ is Schwartz space and  $\Vert\cdot\Vert_{1,p}$
	is the norm in the Bessel potential space $L^{1,p}(\R^d)=\{f:\;f=G_1\star g,\;g\in L_p(\R^d)\}$, which is defined by $\Vert f\Vert_{1,p}=\Vert g\Vert_p$.
	By the Calderon's Theorem, $L^{1,p}(\R^d)$ coincides with the Sobolev space
	$W_p^1(\R^d)$ and the norm $\Vert\cdot\Vert_{1,p}$ is equivalent
	to the norm $\Vert\cdot\Vert_{W_p^1}$ (\cite{AH}, p. 13). These
	circumstances and definition \eqref{dfWincap} of the harmonic capacity
	imply that there is $C>0$ such that for any compact set
	$E\subset\R^d$ the estimate is valid for $p=2$:
	\begin{eqnarray*}
		\hskip-5mm&&C_{G_1,2}(E)\ge\\
		\hskip-5mm&&C\,\inf\Big(\big\{\Vert u\Vert_{W_2^1}^2=\int_{\R^d}\big(|\nabla
		u(\e)|^2+|u(\e)|^2\big)\, \mathrm{d}\e:\;u\in\mathcal{S}, \;u\ge
		1\;\mathrm{on}\;E\big\}\Big)\ge\\
		\hskip-5mm&&C\,\inf\Big(\big\{\int_{\R^d}|\nabla u(\e)|^2\,
		\mathrm{d}\e\,:\;u\in\mathcal{S}, \;u\ge
		1\;\mathrm{on}\;E\big\}\Big)=C\,\mathrm{cap}(E).
	\end{eqnarray*}
	Clain (i) is proven.
	
	(ii) By the arguments of the previous claim, $f\in W_p^1(\R^d)$.  Since $p>d$, then  by the well known Sobolev's theorem, the space $W_p^1(\R^d)$ is embedded continuously into the space $C_B^0(\R^d)$ of all continuous and bounded functions on $\R^d$ (\cite{AF}, Chapt. 4). Claim (ii) is proven,
	\end{proof}

The proof of following claim is
analogous to the first part of the proof of Theorem 2.4.1 from
\cite{AH} (p. 30):

\begin{lemma}\label{lminrchanginfsup}
	Suppose that $X$ and $Y$ are sets
	and $F:\,X\times Y\rightarrow\R$ is a function.
	Then the inequality is
	valid:
	\begin{equation}\label{infsup}
	\inf_{x\in X}\Big(\sup_{y\in Y} F(x,y)\Big)\ge\sup_{y\in
		Y}\Big(\inf_{x\in X}F(x,y)\Big).
	\end{equation}
	\begin{proof}
		It is clear that $\sup_{y\in Y} F(x,y)\ge F(x,w)$ for any $x\in X$
		and $w\in Y$. Hence $\inf_{x\in X}\Big(\sup_{y\in Y}
		F(x,y)\Big)\ge\inf_{x\in X}F(x,w)$ for any $w\in Y$. Therefore the
		desired inequality \eqref{infsup} is valid. The lemma is proven.
	\end{proof}
\end{lemma}

\subsection{Proof of Theorem \ref{thdoublrear}}
\begin{proof} Let us take $\theta>1$ and denote 
	\begin{equation*}
	\tilde\gamma(r)=\gamma(r)/\theta,\quad \sigma(r)=(1-\tilde\gamma(r))\mu(\mathcal{G}_r(0)).
	\end{equation*}
	It is clear that the function $\tilde\gamma(r)$ satisfies
	conditions \eqref{cndgammar}.
	 Since $\mathrm{BP}_{eq}(\by,r)\subseteq\mathrm{BP}(\by,r)$, we have:
	\begin{eqnarray}\label{ineqforsup}
	\hskip-10mm&&\sup_{\mu\,\in\,
		\mathrm{BP}(\by,r)}\;\inf_{F\in\mathcal{M}^\mu_{\tilde\gamma(r)}(\by,r)}\int_{\bar{\mathcal G}_r(\by)}\Big(\int_{\mathcal{G}_r(\by)\setminus
		F}G_1(\e-\bs)\sqrt{V(\bs)}\,\mathrm{d}\bs\Big)^2\mu( \mathrm{d}\e)\nonumber\ge\\
	\hskip-10mm&&\sup_{\mu\,\in\,
		\mathrm{BP}_{eq}(\by,r)}\;
	\inf_{F\in\mathcal{M}^\mu_{\tilde\gamma(r)}(\by,r)}\int_{\bar{\mathcal G}_r(\by)}
	\Big(\int_{\mathcal{G}_r(\by)\setminus
		F}G_1(\e-\bs)\sqrt{V(\bs)}\,\mathrm{d}\bs\Big)^2\mu(\mathrm{d}\e)
	\end{eqnarray}
	Let us take $\mu\in\mathrm{BP}_{eq}(\by,r)$. Then $\mu$ is equivalent to the Lebesgue measure and, in view of definition \eqref{defcore} of the base polyhedron, it is finite.  
	Since we assume that $V\in L_{p,\,loc}(\R^d)$ with
	$p>d/2$, then in view of claim (ii) of Lemma \ref{lmC12eqWien}, the function in the brackets under 
	integral of \eqref{ineqforsup} is continuous and bounded on the domain $\mathcal{G}_r(\by)$. hence this function is $\mu$-integrable in $\mathcal{G}_r(\by)$, i.e., the above integral is finite. Let us represent:
	\begin{eqnarray*}
		\hskip-10mm&&\int_{\mathcal{G}_r(\by)}\mu( \mathrm{d}\e)\Big(\int_{\mathcal{G}_r(\by)\setminus
			F}G_1(\e-\bs)\sqrt{V(\bs)}\,\mathrm{d}\bs\Big)^2=\\
		\hskip-10mm&&\int_{\mathcal{G}_r(\by)}\mu( \mathrm{d}\e)\int_{(\mathcal{G}_r(\by)\setminus
			F)\times(\mathcal{G}_r(\by)\setminus
			F)}G_1(\e-\bs)\sqrt{V(\bs)}
		G_1(\e-\bt)\sqrt{V(\bt)}\,\mathrm{d}\bs\,\mathrm{d}\bt.
		\nonumber
	\end{eqnarray*}
	It is easy to see that the function under integrals in the right hand side of the last formula is 
	$\mu\times\mathrm{mes}_{2d}$-measurable. Then by Fubini Theorem  (\cite{Hal}, Sec. 36, p. 147, Theorem B), we obtain that
\begin{eqnarray}\label{changord}
\hskip-10mm&&\int_{\mathcal{G}_r(\by)}\mu( \mathrm{d}\e)\Big(\int_{\mathcal{G}_r(\by)\setminus
	F}G_1(\e-\bs)\sqrt{V(\bs)}\,\mathrm{d}\bs\Big)^2=\nonumber\\
	\hskip-10mm&&\int_{(\mathcal{G}_r(\by)\setminus F)\times(\mathcal{G}_r(\by)\setminus
	F)}\sqrt{V(\bs)}\sqrt{V(\bt)}K_\mu(\bs,\bt)\,\mathrm{d}\bs\,\mathrm{d}\bt
\end{eqnarray}
and the function under the last integral belongs to the space $L_1(\mathcal{G}_r(\by)\times\mathcal{G}_r(\by))$.	
	Hence the function $X_\mu(\bs,\bt)$, defined by \eqref{dfWbsbt}, \eqref{dfKmubsbt}, belongs to the space\\ $L_1(\mathcal{G}_r(\by)\times\mathcal{G}_r(\by),\;\mu\times\mu)$.
	Therefore this function has the partial nonincreasing rearrangement $X_\mu^\star(\cdot,\bt)(u;\,\by,r;\,\mu)\in L_1(\mathcal{G}_r(\by),\,\mu)$ and the repeated nonincreasing rearrangement $(X_\mu^\star)^\star(u;\,\by,r;\,\mu)$ (Definition \ref{dfreprearr} and 
	Proposition \ref{prexistpartrear}). We have:
	\begin{equation*}
	\{E={\mathcal G}_r(\by)\setminus
	F:\;F\in\mathcal{M}_{\tilde\gamma(r)}^\mu(\by,r)\}\subseteq\mathcal{E}_{\sigma(r)}(\by,r) ,
	\end{equation*}
	where $\mathcal{E}_{\sigma(r)}(\by,r) =\mathcal{E}\big(\sigma(r),\,\mathcal{G}_r(\by),\,\mu\big)$.
	Recall that the collection $\mathcal{E}(t,X,\,\mu)$ is defined  in the formulation of Problem \ref{extrprobl} (in our case $\Sigma=\Sigma_L({\mathcal G}_r(\by))$).  Since the Lebesgue measure $\mathrm{mes}_d$ is is non-atomic and finite in ${\mathcal G}_r(\by)$, and each measure from $\mathrm{BP}_{eq}(\by,r)$ is absolute continuous with respect to $\mathrm{mes}_d$, then $\mathrm{BP}_{eq}(\by,r)$ consists 
	of non-atomic measures (\cite{John}, Theorem 2,4).
	We have: 
	\begin{eqnarray}\label{estinf}
		&&\inf_{F\in\mathcal{M}^\mu_{\tilde\gamma(r)}(\by,r)}\int_{(\mathcal{G}_r(\by)\setminus
			F)\times(\mathcal{G}_r(\by)\setminus
			F)}\sqrt{V(\bs)}\sqrt{V(\bt)}K_\mu(\bs,\bt)\,\mathrm{d}\bs\,\mathrm{d}\bt\ge\nonumber\\
		&&\inf_{E\in\mathcal{E}_{\sigma(r)}(\by,r) }\int_{E\times E}\sqrt{V(\bs)}\sqrt{V(\bt)}K_\mu(\bs,\bt)\,\mathrm{d}\bs\,\mathrm{d}\bt\ge\nonumber\\
		&&\inf_{E\in\mathcal{E}_{\sigma(r)}(\by,r)}\;\inf_{G\in\mathcal{E}_{\sigma(r)}(\by,r) }\int_{E}\,\Big(\int_{G}X_\mu(\bs,\bt)\,\mu(\mathrm{d}\bs)\Big)\,\mu(\mathrm{d}\bt).
	\end{eqnarray}
Lemma \ref{lmpontwisesupinf} implies that in the Riesz space $M_\mu({\mathcal G}_r(\by))$ (defined in Section \ref{sec:C}) there exists the lattice infimum $\bigwedge\Big(\{\int_{G}X_\mu(\bs,\bt)\,\mu(\mathrm{d}\bs)\}_{G\in\mathcal{E}_{\sigma(r)}(\by,r)}\Big)$, it belongs to $L_1({\mathcal G}_r(\by),\mu)$ and for any $E\in\Sigma_L({\mathcal G}_r(\by))$
\begin{eqnarray}\label{estbylatinf}
&&\inf_{G\in\mathcal{E}_{\sigma(r)}(\by,r) }\int_{E}\,\Big(\int_{G}X_\mu(\bs,\bt)\,\mu(\mathrm{d}\bs)\Big)\,\mu(\mathrm{d}\bt)\ge\nonumber\\
&&\int_{E}\,\Big(\inf_{G\in\mathcal{E}_{\sigma(r)}(\by,r) }\int_{G}X_\mu(\bs,\bt)\,\mu(\mathrm{d}\bs)\Big)\,\mu(\mathrm{d}\bt).
\end{eqnarray}
Using claim (ii) of Proposition \ref{prexistpartrear} and applying twice Proposition \ref{prsolextrprob} and estimate \eqref{estJR1} (Propositio \ref{lmJR}), we get for $\mu\,\in\,
\mathrm{BP}_{eq}(\by,r)$:
	\begin{eqnarray*}
		&&\inf_{E\in\mathcal{E}_{\sigma(r)}(\by,r) }\int_{E}\,\Big(\inf_{G\in\mathcal{E}_{\sigma(r)}(\by,r) }\int_{G}X_\mu(\bs,\bt)\,\mu(\mathrm{d}\bs)\Big)\,\mu(\mathrm{d}\bt)=\\
		&&\inf_{E\in\mathcal{E}_{\sigma(r)}(\by,r) }\int_{E}J_{X_\mu(\bs,\bt)}(\sigma(r),\by,r,\mu)\,\mu(\mathrm{d}\bt)\ge\\
		&&\frac{(\theta-1)\psi_\mu(r)}{\theta}
		\inf_{E\in\mathcal{E}_{\sigma(r)}(\by,r) }
		\int_{E}X_\mu^\star(\cdot,\bt)(\psi_\mu(r);\, \by,r,\nu)\,\mu(\mathrm{d}\bt)=\\
		&&\frac{(\theta-1)\psi_\mu(r)}{\theta}J_{X_\mu^\star(\cdot,\bt)}(\sigma(r);\,\by,r;\,\mu)\ge\\
		&&\Big(\frac{(\theta-1)\psi_\mu(r)}{\theta}\Big)^2
		(X_\mu^\star)^\star(\psi_\mu(r);\,\by,r;\,\mu),
	\end{eqnarray*}
	where $\psi_\mu(r)=\gamma(r))\mu(\mathcal{G}_r(0))$. This estimate, inequality \eqref{ineqforsup}, equality \eqref{changord}, estimates \eqref{estinf} and \eqref{estbylatinf}, condition 
	\eqref{cnddoublrear} and Theorem \ref{thusecore} imply the desired claim. Theorem \ref{thdoublrear} is proven. 
\end{proof}	

\subsection{Proof of Corollary \ref{crdoublrear}}

\begin{proof}
	 The inclusion  $M_f(\by,r)\subseteq BP_{eq}(\by,r)$ (Proposition \ref{lmdescribedens}) implies:
	\begin{equation*} 
	\sup_{\mu\,\in\,
		\mathrm{BP}_{eq}(\by,r)}\;(X_\mu^\star)^\star(\psi(r);\,\by,r;\,\mu)\ge \sup_{\mu\,\in\,
		M_f(\by,r)}\;(X_\mu^\star)^\star(\psi(r);\,\by,r;\,\mu).
	\end{equation*}
	Then in view of Theorem \ref{thdoublrear}. we obtain the desired claim.
\end{proof}

\subsection{Proof of Theorem \ref{thdoubrearrleb}}

\begin{proof}
	Let us take a ball $B_r(\by)$. 
	Consider the function $f(t)=t^{(d-2)/d}$ and the absolutely continuous measure $\mu_s$ on $B_r(\by)$, defined by \eqref{dfmusnewest}, \eqref{dfP1newest} and \eqref{dfmdr}. As we have noticed in the proof of Theorem \ref{thcondlebesgue},
	this measure belongs to $M_f(r,\by)$. 
	Consider the function 
	\begin{eqnarray}\label{dfZmusnewestleb}
	&&X_{\mu_s}(\e,\bt)=\sqrt{V(\e)}\sqrt{V(\bt)}
	\frac{\mathrm{d}\mathrm{mes}_d}{\mathrm{d}\mu_s}(\e)\frac{\mathrm{d}\mathrm{mes}_d}{\mathrm{d}\mu_s}(\bt)K_{\mu_s}(\e,\bt)=\nonumber\\
	&&\sqrt{V(\e)}\sqrt{V(\bt)}\big(f^\prime(s_{r,\by}(\e))\big)^{-1}\big(f^\prime(s_{r,\by}(\bt))\big)^{-1}
	K_{\mu_s}(\e,\bt)=\\
	&&d^2/(d-2)^2\sqrt{V(\e)}\sqrt{V(\bt)}\big(s_{r,\by}(\e)\big)^{2/d}\big(s_{r,\by}(\bt)\big)^{2/d} K_{\mu_s}(\e,\bt).\nonumber
	\end{eqnarray}
	Recall that the function $K_{\mu_s}(\e,\bt)$ is defined by \eqref{dfKmubsbt} with $\mu=\mu_s$. Using Lemma \ref{lmineqforrearnewest} with $t=\gamma(\tilde r)\;(\tilde r=r/\sqrt{d})$,
	and taking  
	\begin{equation*}
	W(\e)=W_{\mu_s}(\e,\bt)=\sqrt{V(\e)}\sqrt{V(\bt)}\big(f^\prime(s_{r,\by}(\e))\big)^{-1}
	K_{\mu_s}(\e,\bt)
	\end{equation*}
	for a fixed $\bt\in B_r(\by)$, we get that for some $\kappa,\delta\in(0,1)$ 
	there is a cube $Q_{\tilde r}(\tilde\by)\subseteq B_r(\by)\cap Q_1(\vec l(\by,r))$ 
	such that 
	\begin{eqnarray}\label{ineqforrearV1newestleb}
	&& X^\star_{\mu_s}(\cdot,\bt)\big(\bar\gamma(r)\cdot\mu_s(B_r(\by));\,B_r(\by);\,\mu_s\big)\ge
	\nonumber\\
	&&\delta\cdot W^\star_{\mu_s}(\cdot,\bt)\big(\gamma(\tilde r)
	\mathrm{mes}_d(Q_{\tilde r}(0));\,Q_{\tilde r}(\tilde\by)\big),
	\end{eqnarray}
	where $\bar\gamma(r)=\kappa\gamma(r/\sqrt{d})$> 
	Applying 
	Lemma \ref{lmcomparrear} and using 
	again Lemma \ref{lmineqforrearnewest} 
	with
	\begin{equation*}
	W(\bt)=W^\star_{\mu_s}(\cdot,\bt)\big(\hat\gamma(\tilde r,K)
	\mathrm{mes}_d(Q_{\tilde r}(0));\,Q_{\tilde r}(\tilde\by)\big),
	\end{equation*} 
	we get from \eqref{ineqforrearV1newestleb}:
	\begin{eqnarray}\label{ineqforrearV1newestnewleb}
	\hskip-5mm&& (X^\star_{\mu_s})^\star\big(\bar\gamma(r)\cdot\mu_s(B_r(\by));\,B_r(\by);\,\mu_s\big)\ge
	\nonumber\\
	\hskip-5mm&&\delta
	 (W^\star_{\mu_s})^\star\big(\gamma(\tilde r)\mathrm{mes}_d(Q_{\tilde r})(\tilde\by),\bar\gamma(r)\cdot\mu_s(B_r(\by));\;Q_{\tilde r})(\tilde\by),\, B_r(\by);\;\mathrm{mes}_d,\,\mu_s\big)\ge\nonumber\\
	\hskip-5mm&&\delta^2
	(Y_{\mu_s}^\star)^\star\big(\gamma(\tilde r)
	\mathrm{mes}_d(Q_{\tilde r}(0));\,Q_{\tilde r}(\tilde\by)\big),
	\end{eqnarray}
	where
	\begin{eqnarray*}
		&&Y_{\mu_s}(\e,\bt)=\sqrt{V(\e)}\sqrt{V(\bt)}\int_{B_r(\by)}G_1(\e-\bs)G_1(\bs-\bt)\,\mu_s(\mathrm{d}\bs)=\\
		&&\sqrt{V(\e)}\sqrt{V(\bt)}\int_{B_r(\by)}G_1(\e-\bs)G_1(\bs-\bt)f^\prime(s_{r,\by}(\bs))\,
		\mathrm{d}\bs
	\end{eqnarray*}
	Let us notice that  $\min_{t\in[0,1]}f^\prime(t)=(d-2)/d$. 
	Hence in view of Proposition \ref{estconvBeskern},
	$Y_{\mu_s}(\e,\bt)\ge \big((d-2)/d\big)A(r_0) Y(\e,\bt)$ for any $\e,\,\bt\in B_r(\by)$.
	Hence by double use of Lemma \ref{lmcomparrear} we get:
	\begin{eqnarray*}
	&&(Y_{\mu_s}^\star)^\star\big(\gamma(\tilde r)
	\mathrm{mes}_d(Q_{\tilde r}(0));\,Q_{\tilde r}(\tilde\by)\big)\ge\\ 
	&&\big((d-2)/d\big)A(r_0)(Y^\star)^\star\big(\gamma(\tilde r)
	\mathrm{mes}_d(Q_{\tilde r}(0));\,Q_{\tilde r}(\tilde\by)\big).
	\end{eqnarray*}
This inequality, estimate \eqref{ineqforrearV1newestnewleb}, condition \eqref{cnddoubrearrleb}, inclusion
$\mu_s\in M_f(r,\by)$ and Corollary \ref{crdoublrear} imply the desired claim. Theorem \ref{thdoubrearrleb} is proven
\end{proof}

In the proof of Theorem \ref{thdoubrearrleb} we have used the following lemma:

\begin{lemma}\label{lmcomparrear}
	Let $(X,\Sigma,\mu)$ be a measure space. 
	
	(i) If $f(x)$is a non-negative $\mu$-measurable 
	function on $X$ and $C>0$ is constant, then for any $t>0$ $(Cf)^\star(t;\,X;\mu)=C f^\star(t;\,X;\,\mu)$. 
	
	(ii) If $f(x)$, $g(x)$ are non-negative $\mu$-measurable 
	functions on $X$ such that $f(x)\ge g(x)$ for any $x\in X$, then for any $t>0$ $f^\star(t,X,\mu)\ge g^\star(t,X,\mu)$,
\end{lemma}
\begin{proof} (i) We have by definitions \eqref{dfSVyrdel}-\eqref{dfLVsry}:
	\begin{equation*}
	\mathcal{L}^\star(s,Cf,X;\,\mu)=\mathcal{L}^\star(s/C,f;\,X;\,\mu)\quad (s>0),
	\end{equation*}
	hence $\lambda^\star(s;\,Cf;\,X;\,\mu)=\lambda^\star(s/C;f;\,X;\,\mu)$. Therefore for $t>0$
	\begin{eqnarray*}
		&&(Cf)^\star(t;\,X;\,\mu)=\sup\{s>0:\;\lambda^\star(s/C,f,X,\mu)\ge t\}=\\
		&&\sup\{Cu>0:\;\lambda^\star(u,f,X,\mu)\ge t\}=C f^\star(t,X,\mu),
	\end{eqnarray*}
	i.e., the desired equality is valid.
	
	(ii) Since $g(x)\ge s>0$ implies $f(x)\ge s$, we obtain  that $\mathcal{L}^\star(s;\,g;\,X;\,\mu)\subseteq\mathcal{L}^\star(s;\,f;\,X;\,\mu)$, 
	hence  $\lambda^\star(s;\,g;\,X;\,\mu)\le\lambda^\star(s;\,f;\,X;\,\mu)$ and hence the desired inequality is valid.
\end{proof}

\subsection{Proof of Proposition \ref{prrellebesdoublebes}}

\begin{proof}
	Suppose that the conditions of Theorem \ref{thcondlebesgue} are satisfied. Consider the function $Y(\e,\bt)$, defined by \eqref{dfYmust} and estimate it from below for $\e,\bt\in Q_r(\by)$:
	$Y(\e,\bt)\ge (\sqrt{d}r)|^{(2-d)}\sqrt{V(\e)}\sqrt{V(\bt)}$. Hence using Lemma \ref{lmcomparrear} twice and Lemma \ref{lmrerrofpow}, we get:
	\begin{eqnarray*}
	&&(Y^\star)^\star\big(\psi(r);\,Q_r(\by))\big)\ge (\sqrt{d}r)|^{(2-d)}\Big((\sqrt{V})^\star\big(\psi(r);\,Q_r(\by))\big)\Big)^2=\\ &&(\sqrt{d}r)|^{(2-d)}V^\star\big(\psi(r);\,Q_r(\by))\big).
	\end{eqnarray*}
	This estimate and condition \eqref{cndlebesgue} imply that condition \eqref{cnddoubrearrleb} of Theorem \ref{thdoubrearrleb} is fulfilled. Proposition \ref{prrellebesdoublebes} is proven. 
\end{proof}

In the proof of Proposition \ref{prrellebesdoublebes} we have used the following claim:
\begin{lemma}\label{lmrerrofpow}
	Suppose that $\Omega$ is a domain in $\R^d$ and
	$f:\,\Omega\rightarrow\C$ is a measurable nonnegative function.  
	If $\alpha>0$, the
	equality is valid for any $t> 0$:
	$(f^\alpha)^\star(t)=(f^\star(t))^\alpha$.
\end{lemma}
\begin{proof}
	 We have:
	\begin{eqnarray*}
		&&\hskip-5mm\lambda^\star(s)=\mathrm{mes}_d\Big(\{\e\in\Omega:\,(f(\e))^\alpha\ge s\}\Big)=\mathrm{mes}_d\Big(\{\e\in\Omega:\,f(\e)\ge s^{1/\alpha}\}\Big)=\\
		&&\hskip-5mm\lambda^\star(s^{1/\alpha}),
	\end{eqnarray*}
	hence
	\begin{eqnarray*}
		(f^\alpha)^\star(t)=\sup\{s> 0:\,\lambda^\star(s^{1/\alpha})\ge t\}=\sup\{u^\alpha
		:\,\lambda^\star(u)\ge t\}=(f^\star(t))^\alpha.
	\end{eqnarray*}
\end{proof}

\subsection{Proof of Theoren \ref{thlogmdensdoubrearr}}

\begin{proof}
	Let us take a cube $Q_r(\by)$.  It is clear that there 
	are $\vec l(\by,r)\in\Z^d$ and a cube $Q_{r/2}(\tilde\by)\subseteq Q_r(\by)\cap Q_1(\vec l(\by,r))$. By Lemma \ref{prestbarVstarbelownewest} with $W(\e)=Y(\e,\bt)$ for a fixed $\bt\in Q_{r/2}(\tilde\by)$ and Lemma \ref{lminclrear2}, there are $j\in\{1,2,\dots,n\}$ with 
	\begin{equation}\label{dfnlogtildr}
	n=\big[\log_m\big(\frac{2}{\theta r}\big)\big]+2
	\end{equation}
	and $K\in(0,1)$, such that 
	\begin{eqnarray}\label{ineqforrearV1newest}
\hskip-5mm&& Y^\star(\cdot,\bt)\big(\bar\gamma(r)\cdot\mathrm{mes}_d(Q_r(0));\,Q_r(\by)\big)\ge\nonumber\\ 
\hskip-5mm&&Y^\star(\cdot,\bt)\big(\hat\gamma(r/2,K)\cdot\mathrm{mes}_d(Q_{r/2}(0));\,Q_{r/2}(\tilde\by)\big)\ge
	\nonumber\\
\hskip-5mm&&\min_{\vec\xi:\,Q(\vec\xi,\,n)\subseteq F_j} Y^\star(\cdot,\bt)\big(\psi(m,n);\,Q(\vec\xi,n)\big),
	\end{eqnarray}
	where 
	$\hat\gamma(\rho,K)=K\gamma(\theta \rho/m^2)\theta^d$,
\begin{equation*}
\psi(m,n)=\gamma(m^{-n})\mathrm{mes}_d(Q(\vec 0,n)),\quad \bar\gamma(r)=2^{-d}\hat\gamma(r/2,\,K)
\end{equation*} 
 and $F_j$ is a non-empty union of cubes $Q(\vec\xi,\,n)$ with $\vec\xi\in m^{-n}\cdot\Z^d$, such that $F_j\subseteq Q_{r/2}(\tilde\by)\cap D_j(\vec l(\by,r))$. Notice that since the function $\gamma(r)$ satisfies conditions \eqref{cndgammar}, the function $\bar\gamma(r)$ satisfies this condition too for some $r_0>0$.
	Applying Lemma \ref{lmcomparrear} and using again Lemma  \ref{prestbarVstarbelownewest}  with
	\begin{equation*}
	W(\bt)=Y^\star(\cdot,\bt)\big(\psi(m,n);\,Q(\vec\xi,n)\big)
	\end{equation*} 
and \ref{lminclrear2},  we get from \eqref{ineqforrearV1newest}:
	\begin{eqnarray}\label{ineqforrearV1newestnew}
	\hskip-5mm&& (Y^\star)^\star\big(\bar\gamma(r)\cdot\mathrm{mes}_d(Q_r(\by));\,Q_r(\by)\big)\ge\nonumber
	\\
	\hskip-5mm&&\min_{\vec\xi:'\,Q(\vec\xi,\,n)\subseteq F_j} (Y^\star)^\star\big(\psi(m,n),\bar\gamma(r)\cdot\mathrm{mes}_d(Q_r(0));\,Q(\vec\xi,n),\,Q_r(\by\big)\ge\nonumber\\
	\hskip-5mm&&\min_{\vec\xi:\,Q(\vec\xi,\,n)\subseteq F_j} (Y_{\mu_s}^\star)^\star\big(\psi(m,n)
	,\hat\gamma(r/2,K)
	\mathrm{mes}_d(Q_{ r/2}(0));\,Q(\vec\xi,n),\,Q_{r/2}(\tilde\by)\big)\ge\nonumber\\
	\hskip-5mm&&\min_{\vec\xi,\eta:\,Q(\vec\xi,\,n)\subseteq F_j,\,Q(\vec\eta,\,n)\subseteq F_j}(Y_{\mu_s}^\star)^\star\big(\psi(m,n);\,Q(\vec\xi,n),\,Q(\vec\eta,n)\big).
	\end{eqnarray}
Let us notice that by 
	\eqref{dfnlogtildr}, $r\le\frac{2}{\theta}m^{-(n-2)}$. Hence if $Q(\vec\xi,\,n)\subseteq F_j$ and $Q(\vec\eta,\,n)\subseteq F_j$, then, since $F_j\subseteq Q_{r/2}(\tilde\by)$,  we have: $|\vec\xi-\vec\eta|\le \frac{\sqrt{d}}{\theta}m^{-(n-2)}$. Therefore 
	in view of condition \eqref{Xinonempty2}, estimate \eqref{ineqforrearV1newestnew}  implies:
	\begin{eqnarray*}
		&&\hskip-12mm (Y^\star)^\star\big(\bar\gamma(r)\cdot\mathrm{mes}_d(Q_r(\by));\,Q_r(\by)\big)\ge\nonumber\\
		&&\hskip-12mm \min_{\vec\xi,\,\vec\eta\in\bigcup_{j=1}^n\Xi_{n}(\vec l,\,j),\;|\vec\xi-\vec\eta|\le\frac{\sqrt{d}}{\theta}m^{-(n-2)}} (Y^\star)^\star\big(\psi(m,n);\,Q(\vec\xi,n),\,Q(\vec\eta,n)\big).
	\end{eqnarray*}  
	Since $Q_r(\by)\cap Q_1(\vec l(\by,r))\neq\emptyset$, this estimate and 
	condition \eqref{cndlogdensdoubrearr} imply that condition \eqref{cnddoubrearrleb}
	of Theorem \ref{thdoubrearrleb} is satisfied with $\gamma(r)=\bar\gamma(r)$. 
	Then the spectrum of the operator $H=-\Delta+V(\e)\cdot$ is discrete. Theorem \ref{thlogmdensdoubrearr} is proven. 
\end{proof}

\section{Some examples} \label{sec:examples}
\setcounter{equation}{0}
First of all, consider some examples of the $(\log_m,\,\theta)$-dense system (see Definition \ref{dfdenslogmtetpart1}).

\begin{example}\label{ex1}
	Consider the classical middle third Cantor set $\mathcal{C}\subset[0,1]$, Let $I_{n,k}\;(n=1,2,\dots),\,k=1,2,\dots, 2^{n-1}$ be the closures 
	of intervals adjacent to $\mathcal{C}$. It is known that they are disjoint and for any fixed $n$ and each $k\in\{1,2,\dots, 2^{n-1}\}$ 
	$\mathrm{mes}_1(I_{n,k})=3^{-n}$. For fixed $n$ we shall number the intervals $I_{n,k}$ from the left to the right. Denote $D_n=\bigcup_{k=1}^{2^{n-1}}I_{n,k}$. 
	As we have shown in \cite{Zel1} (Example 5.1),  the sequence $\{D_n\}_{n=1}^\infty$ forms 
	in $[0,1]$ a $(\log_3,\,1/9)$-dense system. 
\end{example}

\begin{example}\label{ex2}
	Consider a cube $Q_1(\by)\subset\R^d$, represented in the form $Q_1(\by)=Q_1(\by_1)\times Q_1(\by_2)$, where 
	$Q_1(\by_1)\subset\R^{d_1}$ and $Q_1(\by_2)\subset\R^{d_2}$. Let $\{D_n\}_{n=1}^\infty$ be a sequence of subsets of 
	the cube $Q_1(\by_1)$, forming in it a $(\log_m,\theta)$-dense system. It is easy to see that the sequence
	$\{D_n\times Q_1(\by_2)\}_{n=1}^\infty$ forms in $Q_1(\by)$ a $(\log_m,\theta)$-dense system too.	
\end{example}

Now we shall construct a counterexample connected with conditions of discreteness of the spectrum
of the operator $H=-\Delta+V(\e)\cdot$, obtained above.

\begin{example}\label{exdoubrearlebes}
	By Proposition \ref{prlogmdensdoubrearr}. if  $V\in L_{p,\,loc}(\R^d)$ with
	$p>d/2$, then Theorem \ref{thcondlebesgue} implies Theorem \ref{thdoubrearrleb}.
	On the other hand, we see from the proof of Theorem \ref{thlogmdensdoubrearr} that it follows from Theorem \ref{thdoubrearrleb}.
	We shall construct an example of the potential $V\in L_{\infty,\,loc}(\R^d)$, for which
	the conditions of Theorem \ref{thlogmdensdoubrearr} are satisfied (hence the
	spectrum of the operator $H=-\Delta+V(\e)\cdot$ is discrete), but
	condition \eqref{cndlebesgue} of Theorem \ref{thcondlebesgue} is not satisfied for it.
	Let us return to the sequence $\{D_n\}_{n=1}^\infty$ of subsets of the interval $[0,\,1]$ and considered in Example \ref{ex1}, and the following 
	sequence of subsets of the cube $Q_1(0)$:
	\begin{equation}\label{dfcalDnew}
	\mathcal{D}_n=D_n\times[0,\,1]^{d-1}.
	\end{equation} 
	Consider also the translations of the cube $Q_1(0)$ and the sets
	$\mathcal{D}_n$ by the vectors $\vec l=(l_1,l_2,\dots,l_d)\in\Z^d$: $Q_1(\vec l)=Q_1(0)+\vec l$, 
	\begin{equation}\label{dfcalDveclnew}
	\mathcal{D}_n(\vec l)=\mathcal{D}_n+\vec l.
	\end{equation}
	The arguments of Examples \ref{ex1} and \ref{ex2} imply that for any fixed $\vec l\in\Z^d$  the sequence $\{\mathcal{D}_n(\vec l)\}_{n=1}^\infty$ 
	forms in $Q_1(\vec l)$ a $(\log_3,\,1/9)$-dense system.  For $\beta\in[0,1]$ consider on $\R$ the
	$1$-periodic function $\theta_\beta(x)$, defined on the interval
	$[0,1]$ in the following manner:
	\begin{equation}\label{dfthatbetnewest}
	\theta_\beta(x)=\left\{\begin{array}{ll}
	1&\quad\mathrm{for}\quad x\in(0,\beta],\\
	0&\quad\mathrm{for}\quad x\in(\beta,1].
	\end{array}\right.
	\end{equation}
	Let us take
	\begin{equation}\label{alphainnew}
	\alpha\in\Big(0,\,2\Big).
	\end{equation} 
	Consider the following function, defined on $(0,1]$:
	\begin{eqnarray}\label{dfSigmaaNalpdoub}
	&&\Sigma_{N,\,p,\alpha}(x):=\\
	&&\left\{\begin{array}{ll}
	0&\quad\mathrm{for}\quad x\in[0,\,1]\setminus\bigcup_{n=1}^\infty D_n,\\
	N\theta_\beta(3^px)\vert_{\beta=3^{-\alpha n}}C(\theta)3^{-n(d-2)}&\quad\mathrm{for}\quad x\in
	D_n\;(n=1,2,\dots),
	\end{array}\right.\nonumber
	\end{eqnarray}
	where $N>0,\,p\in\N$ and $C(\theta)=\Big(2\sqrt{d}(1+9/d)\Big)^{d-2}$. Recall that I denote by  $P_1$ the operator, defined  by \eqref{dfP1newest}.
	Consider a function
	$\mathcal{N}:\Z^d\rightarrow\R_+$, satisfying the conditions
	\begin{equation}\label{cndNl1newest}
	\mathcal{N}(\vec l)\ge 1,\quad\mathcal{N}(\vec
	l)\rightarrow\infty\quad\mathrm{for}\quad |\vec
	l|_\infty\rightarrow\infty,
	\end{equation}
	\begin{equation}\label{growthNl}
	\sup_{\vec l\,\in\,\Z^d}\frac{\mathcal{N}(\vec
		l)}{3^{|\vec
			l|_\infty(d-2)}}<\infty
	\end{equation}
	where $|\vec l|_\infty=\max_{1\le i\le d}|l_i|$.
	Let us construct the
	desired potential in the following manner:
	\begin{eqnarray}\label{dfpotentValdoub}
	&&V_\alpha(\e):=\Sigma_{N,\,p,\,\alpha}(P_1(\e-\vec
	l))\vert_{N=\mathcal{N}(\vec l),\,p=|\vec l|_\infty+1}\quad
	\mathrm{for}\quad\vec l\in \Z^d\nonumber\\
	&&\mathrm{and}\quad \e\in
	Q_1(\vec l).
	\end{eqnarray}
	It is clear that $V_\alpha\in L_{\infty,\,loc}(\R^d)$.  
	Let us prove that the potential $V_\alpha(\e)$ satisfies all the conditions of Theorem \ref{thlogmdensdoubrearr}. 
	Let us take 
	a natural $n$, 
	\begin{equation}\label{jin1nnew}
	j\in\{1,2,\dots,n\}
	\end{equation}
	and cubes  
	\begin{equation}\label{inclcubenew}
	Q(\vec\xi,\,n),\, Q(\vec\eta,\,n)\subseteq\mathcal{D}_j(\vec l)
	\end{equation} 
	of the $3$-adic partition of $Q_1(\vec l)$ such that $|\vec\xi-\vec\eta|\le\frac{2\sqrt{d}}{\theta}3^{-(n-2)}$.
	Then in view of definition \eqref{dfYmust} of $Y(\e,\bt)$ with $V(\e)=V_\alpha(\e)$,
	\begin{eqnarray}\label{estYdoubstar}
	&&(Y^\star)^\star\big(3^{-\alpha n}\mathrm{mes}_d\big(Q(\vec\xi,\,n)\big);\,Q(\vec\xi,n),\,Q(\vec\eta,n)\big)\ge\nonumber\\
	&&(\sqrt{V_\alpha})^\star\big(3^{-\alpha n}\mathrm{mes}_d\big(Q(\vec\xi,\,n)\big);\,Q(\vec\xi,n)\big)\times\\
	&&(\sqrt{V_\alpha})^\star\big(3^{-\alpha n}\mathrm{mes}_d\big(Q(\vec\xi,\,n)\big);\,Q(\vec\eta,n)\big)
	\frac{3^{n(d-2)}}{C(\theta)}.\nonumber
	\end{eqnarray}
	Let us notice that 
	\begin{equation*}
	P_1\big(Q(\vec\xi,\,n)\big)=[k\,3^{-n},\,(k+1)\,3^{-n}]
	\end{equation*}
	for some $k\in\Z$.
	Then taking into account definitions \eqref{dfcalDnew}-\eqref{dfSigmaaNalpdoub}, \eqref{dfpotentValdoub} and the $1$-periodicity of $\theta_\beta(t)$, we get for $|\vec l|_\infty>n$ and 
	\begin{equation}\label{bounsfors}
	0<s\le \sqrt{\frac{C(\theta)\mathcal{N}(\vec l)}{3^{n(d-2)}}}:
	\end{equation}
	\begin{eqnarray}\label{bigestnew}
	&&\hskip-12mm\mathrm{mes}_d\Big(\big\{\e\in Q(\vec\xi,n):\,\sqrt{V_\alpha(\e)}\ge s\big\}\Big)=\nonumber\\
	&&\hskip-12mm3^{-(d-1)n}\mathrm{mes}_1\Big(\big\{x\in[k\,3^{-n},\,(k+1)\,3^{-n}]:\,\nonumber\\
	&&\hskip-12mm\theta_\beta(3^p\,x)\vert_{\beta=3^{-\alpha j},\,p=|\vec l|_\infty+1}\ge s^2\frac{3^{n(d-2)}}{C(\theta)\mathcal{N}(\vec l)} \big\}\Big)=\nonumber\\
	&&\hskip-12mm3^{-(d-1)n}\mathrm{mes}_1\Big(\big\{x\in[0,\,3^{-n}]:\,\theta_\beta(3^p\,x)\vert_{\beta=3^{-\alpha j},\,p=|\vec l|_\infty+1}>0 \big\}\Big)=\nonumber\\
	&&\hskip-12mm\frac{3^{-(d-1)n}}{3^{|\vec l|_\infty+1}}\mathrm{mes}_1\Big(\big\{t\in[0,\,3^{|\vec l|_\infty+1-n}]:\,\theta_\beta(t)\vert_{\beta=3^{-\alpha j}}>0\big\}\Big)=\nonumber\\
	&&\hskip-12mm 3^{-dn}\mathrm{mes}_1\Big(\big\{t\in[0,\,1]:\,\theta_\beta(t)\vert_{\beta=3^{-\alpha j}}>0\big\}\Big)=\nonumber\\
	&&\hskip-12mm3^{-\alpha j}\mathrm{mes}_d\big(Q(\vec\xi,\,n)\big)\ge 3^{-\alpha n}\mathrm{mes}_d\big(Q(\vec\xi,\,n)\big).\nonumber
	\end{eqnarray}
	The analogous estimate we get with $\vec\eta$ instead of $\vec\xi$. Therefore in view of \eqref{bounsfors} and definitions \eqref{dfSVyrdel}-\eqref{dfLVsry}, 
	\begin{equation*}
	(\sqrt{V_\alpha})^\star\big(3^{-\alpha n}\mathrm{mes}_d\big(Q(\vec\xi,\,n)\big);\,Q(\vec\xi,\,n)\big)\ge\sqrt{\frac{C(\theta)\mathcal{N}(\vec l)}{3^{n(d-2)}}},
	\end{equation*}
	and
	\begin{equation*}
	(\sqrt{V_\alpha})^\star\big(3^{-\alpha n}\mathrm{mes}_d\big(Q(\vec\eta,\,n)\big);\,Q(\vec\eta,\,n)\big)\ge\sqrt{\frac{C(\theta)\mathcal{N}(\vec l)}{3^{n(d-2)}}},
	\end{equation*}
	if conditions \eqref{jin1nnew} and \eqref{inclcubenew} are satisfied. Then, in view of \eqref{estYdoubstar},
	\begin{equation*}
	(Y^\star)^\star\big(3^{-\alpha n}\mathrm{mes}_d\big(Q(\vec\xi,\,n)\big);\,Q(\vec\xi,n),\,Q(\vec\eta,n)\big)\ge)\mathcal{N}(\vec l).
	\end{equation*}
	This estimate and conditions \eqref{alphainnew}, \eqref{cndNl1newest}-b imply that condition 
	\eqref{cndlogdensdoubrearr} of Theorem \ref{thlogmdensdoubrearr} is satisfied for the potential $V_\alpha(\e)$ with $\gamma(r)=r^\alpha$ satisfying condition \eqref{cndgammar}. Hence the spectrum of the operator 
	$H=-\Delta+V_\alpha(\e)\cdot$ is discrete. Let us show that	condition \eqref{cndlebesgue} of Theorem \ref{thcondlebesgue} is not satisfied for the potential $V_\alpha(\e)$. To this end it is sufficient to show that for any function $\gamma(r)$, satisfying condition \eqref{cndgammar}, there are sequences of numbers $r_j>0$ and points $\by_j\in\R^d$ such that
	\begin{eqnarray}\label{findseqlebes}
	&&\lim_{j\rightarrow\infty}r_j=0,\quad  \lim_{j\rightarrow\infty}|\by_j|=\infty\nonumber\\
	&&\mathrm{and}\quad \limsup_{j\rightarrow\infty} V_\psi^\star\big(\gamma(r_j)\mathrm{mes}_d(Q_{r_j}(\by_j));\,Q_{r_j}(\by_j)\big)<\infty.
	\end{eqnarray}
	Consider 
	the intervals $I_{j,1}=[a_{j,1},\,b_{j,1}]\subset D_{j}\;(j\in\N)$ and
	the cubes $Q_{3^{-j}}(\by_{j})$, where $\by_{j}=\big(a_{j,1}+j,\,0,\dots,0\big)$. Then definitions \eqref{dfSigmaaNalpdoub}, \eqref{dfpotentValdoub} and condition \eqref{growthNl} imply that
	\begin{equation*}
	\sup_{j\in\N,\,\e\in Q_{3^{-j}}(\by_{j})}V_\alpha(\e)<\infty.
	\end{equation*} 
	Hence conditions \eqref{findseqlebes} are satisfied with $r_j=3^{-j}$.
\end{example}

\appendix
\section{Base polyhedron and Choquet integral for harmonic capacity}
\label{sec:A}

\begin{proposition}\label{prcore}
	Let $\Omega$ be an open and bounded subset of $\R^d$.
	
	(i) The base polyhedron $\mathrm{BP}(\bar\Omega))$ of the harmonic
	capacity on $\bar\Omega$, defined by \eqref{defcore},  is
	nonempty, convex and weak*-compact;
	
	(ii) If $F:\,\bar\Omega\rightarrow\R$ is a
	non-negative bounded Borel function,
	then  Choquet integral of it over $\bar\Omega$ by the harmonic capacity
	\begin{equation*}
	\int_{\bar\Omega}F(\e)\mathrm{cap}(\mathrm{d}\,\e):=\int_0^\infty\mathrm{cap}\Big(\big\{\e\in\bar\Omega\,:\;F(\e)\ge
	t\}\Big)\,dt
	\end{equation*}
	is represented in the following manner:
	\begin{eqnarray}\label{reprChoqintWin}
	\int_{\bar\Omega}F(\e)\,\mathrm{cap}(\mathrm{d}\,\e)=\max_{\mu\,\in\,\mathrm{BP}(\bar\Omega))}\int_{\bar\Omega}F(\e)\,\mu(
	\mathrm{d}\e).
	\end{eqnarray}
\end{proposition}\textbf{}
\begin{proof}
	(i) As it is known (\cite{Maz1})(p. 537), the harmonic capacity  on $\Omega$ is a
	positive, monotone, bounded, continuous from above set function and
	$\mathrm{cap}(\emptyset)=0$. Recall that it is submodular, i.e. property \eqref{submod} is valid.
	Let us consider the set
	function which is called {\it dual} to ``cap'':
	$\mathrm{cap}^\star(A)=\mathrm{cap}(\bar\Omega)-\mathrm{cap}(A^c)$,
	where $A^c=\bar\Omega\setminus A$. It is easy to show that ``cap*'
	is a positive, monotone, bounded, continuous from below set function and
	$\mathrm{cap^\star.}(\emptyset)=0$, but it is {\it supermodular}
	in the sense that for any pair of sets
	$A,\,B\in\Sigma(\bar\Omega)$ $\mathrm{cap}^\star(A\cup
	B)+\mathrm{cap}^\star(A\cap
	B)\ge\mathrm{cap}^\star(A)+\mathrm{cap}^\star(B)$. As it is known (\cite{Mar-Mon}, Proposition 1)  
	for the set functions of this kind
	the collection of measures
	\begin{eqnarray}\label{dfCorest}
	&&\mathrm{Core}^\star(\bar\Omega))=\{\mu\in M(\bar\Omega)):\;\\
	&&\mu(A)\ge \mathrm{cap}^\star(A)\;\mathrm{for\; all}\;
	A\in\Sigma(\bar\Omega))\;\mathrm{and}\;\mu(\bar\Omega)=\mathrm{cap}^\star(\bar\Omega)\}\nonumber
	\end{eqnarray}
	is nonempty, convex and and weak*-compact. It is called the {\it
		core} of the set function``cap*''.  Since ``cap*'' is positive, it
	is clear that $\mathrm{Core}^\star(\bar\Omega)\subseteq
	M^+(\bar\Omega)$. It is not difficult to show that
	$\mathrm{BP}(\bar\Omega)=\mathrm{Core}^\star(\bar\Omega))$
	(\cite{Fuj}). Claim (i) is proven.
	
	(ii) we have for any $\mu\in\mathrm{BP}(\bar\Omega)$:
	\begin{eqnarray*}
		&&\int_{\bar\Omega} F(\e)\mu(\,
		\mathrm{d}\e)=\int_0^\infty\mu\Big(\{\e\in\bar\Omega\,:\;F(\e)\ge
		t\}\Big)\,dt\le\\
		&&\int_0^\infty\mathrm{cap}\Big(\{\e\in\bar\Omega\,:\;F(\e)\ge
		t\}\Big)\,dt=\int_{\bar\Omega}F(\e)\,\mathrm{cap}(\mathrm{d}\,\e).
	\end{eqnarray*}
	Hence the inequality
	\begin{eqnarray}\label{ineqintcaogemu}
	\int_{\bar\Omega}F(\e)\,\mathrm{cap}(\mathrm{d}\,\e)\ge\sup_{\mu\,\in\,\mathrm{BP}(\bar\Omega))}\int_{\bar\Omega}F(\e)\,\mu(
	\mathrm{d}\e).
	\end{eqnarray}
	is valid. Let us prove the inverse inequality.  Denote
	$N=\sup_{\e\in\bar\Omega}F(\e)$. Since $F(\e)$ is nonnegative,
	then $N\ge 0$. Consider the function $y=F^\star(\e)$, whose graph
	is symmetric to the graph of the function $y=F(\e)$ with respect
	of the hyperplane $y=N/2$, i.e., $F^\star(\e)=N-F(\e)$. It is
	clear that $F^\star(\e)$ is nonnegative and
	$\sup_{\e\in\bar\Omega}F^\star(\e)=N$. It is known that(\cite{Schm} (Proposition 3)):
	\begin{equation*}
	\int_{\bar\Omega}F^\star(\e)\mathrm{cap}^\star(
	\mathrm{d}\e)=\min_{\mu\in\mathrm{Core}^\star(\bar\Omega)}\int_{\bar\Omega}F^\star(\e)\mu(
	\mathrm{d}\e).
	\end{equation*}
	Since $\mathrm{Core}^\star(\bar\Omega)=\mathrm{BP}(\bar\Omega)$,
	this fact implies that there exists a measure
	$\mu_0\in\mathrm{BP}(\bar\Omega)$ such that
	\begin{equation}\label{equalint}
	\int_{\bar\Omega}F^\star(\e)\,\mathrm{cap}^\star(
	\mathrm{d}\e)=\int_{\bar\Omega}F^\star(\e)\,\mu_0( \mathrm{d}\e).
	\end{equation}
	Let us calculate:
	\begin{eqnarray}\label{calcintcapstr}
	&&\hskip-10mm\int_{\bar\Omega}F^\star(\e)\,\mathrm{cap}^\star( \mathrm{d}\e)=\int_0^N\Big(\mathrm{cap}(\bar\Omega)-\mathrm{cap}\big(\big\{\e\in\bar\Omega\,:\;F^\star(\e)<t\big\}\big)\Big)\,dt=\nonumber\\
	&&\hskip-10mm\int_0^N\Big(\mathrm{cap}(\bar\Omega)-\mathrm{cap}\big(\big\{\e\in\bar\Omega\,:\;F(\e)>N-t\big\}\big)\Big)\,dt=\nonumber\\
	&&\hskip-10mm\int_0^N\Big(\mathrm{cap}(\bar\Omega)-\mathrm{cap}\big(\big\{\e\in\bar\Omega\,:\;F(\e)>s\big\}\big)\Big)\,ds=\nonumber\\
	&&\hskip-10mm
	N\mathrm{cap}(\bar\Omega)-\int_{\bar\Omega}F(\e)\,\mathrm{cap}(
	\mathrm{d}\e),
	\end{eqnarray}
	\begin{eqnarray}\label{calcintmu}
	&&\int_{\bar\Omega}F^\star(\e)\,\mu_0( \mathrm{d}\e)=N\mu_0(\bar\Omega)-\int_{\bar\Omega}F(\e)\,\mu_0( \mathrm{d}\e)=\nonumber\\
	&&N\mathrm{cap}(\bar\Omega)-\int_{\bar\Omega}F(\e)\,\mu_0(
	\mathrm{d}\e).
	\end{eqnarray}
	The equalities \eqref{equalint}-\eqref{calcintmu} imply that
	$\int_{\bar\Omega}F(\e)\,\mathrm{cap}(
	\mathrm{d}\e)=\int_{\bar\Omega}F(\e)\,\mu_0( \mathrm{d}\e)$. This
	means that the inequality inverse to \eqref{ineqintcaogemu} is
	valid. Thus, the representation \eqref{reprChoqintWin} is valid.
	Claim (ii) is proven.
\end{proof}


\section{Estimates for the composition of Bessel kernel with itself}
\label{sec:B}

\begin{proposition}\label{estconvBeskern}
	Consider the function 
	\begin{equation*}
	X(\e,\bt,\by,r)=\int_{B_r(\by)}G_1(\e-\bs)G_1(\bs-\bt)\,
	\mathrm{d}\bs,
	\end{equation*}
	where $G_(\e)$ is the Bessel kernel of the order $1$, defined by \eqref{dfBesskern}.
	Then for any $r_0>0$. $r\in(0,r_0]$ and $\e,\bt\in B_r(\by)$ the estimates
	\begin{equation}\label{estXfrombelow}
	X(\e,\bt,\by,r)\ge\frac{A(r_0)}{|\e-\bt|^{d-2}},
	\end{equation}
	\begin{equation}\label{estXfromabove}
	X(\e,\bt,\by,r)\le\frac{B(r_0)}{|\e-\bt|^{d-2}}
	\end{equation}
	are valid with positive constants $A(r_0)$ and $B(r_0)$ depending only on $r_0$.
\end{proposition}
\begin{proof}
	It is known (\cite{AH}, pp, 9-11) that for $|\e|\le r$ the estimates 
	\begin{equation*}
	\frac{C(r_0)}{|\e|^{d-1}}\le G_(\e)\le\frac{D(r_0)}{|\e|^{d-1}}
	\end{equation*}
	are valid with positive constants $C(r_0)$ and $D(r_0) $depending only on $r_0$.
	Then for any $r\in(0,r_0]$ and $\e,\bt\in B_r(\by)$
	\begin{equation}\label{estXxtyr}
	(C(2r_0))^2S(\e,\bt,\by,r)\le X(\e,\bt,\by,r)\le (D(2r_0))^2S(\e,\bt,\by,r),  
	\end{equation}
	where
	\begin{equation*}
	S(\e,\bt,\by,r)=\int_{\bar B_r(\by)}\frac{\mathrm{d}\bs}{|\e-\bs|^{d-1}|\bs-\bt|^{d-1}}.
	\end{equation*}
	Assuming that $\e\neq\bt$, let us make the change of variables in the last integral $\bu=\frac{\bs-\bz}{|\e-\bt|}$ with $\bz=\frac{\e+\bt}{2}$:  
	\begin{equation}\label{dfSxtyr}	
	S(\e,\bt,\by,r)=\frac{1}{|\e-\bt|^{d-2}}\int_{\bar B(\e,\bt,\by,r)}\frac{\mathrm{d}\bu}{|\bu-\e_0|^{d-1}|\bu+\e_0|^{d-1}},
	\end{equation}
	where $\e_0=\frac{\bs+\bz}{|\e-\bt|}=\frac{\e-\bt}{2|\e-\bt|}$ and 
	$\bar B(\e,\bt,\by,r)=\frac{1}{|\e-\bt|}\big(\bar B_r(\by)-\{\bz\}\big)$ is a closed  ball containing the points $\e_0$ and $-\e_0$. We see that $|\e_0|=1/2$. Using Lemma \ref{lmconvhullofcones}, we have:
	$S(\e,\bt,\by,r)\ge\frac{1}{|\e-\bt|^{d-2}}\tilde S(\e_0,\be)$, where 
	\begin{equation*}
	\tilde S(\e_0,\be)=\int_{{\mathcal C}(\e_0,\be)}\frac{\mathrm{d}\bu}{|\bu-\e_0|^{d-1}|\bu+\e_0|^{d-1}}
	\end{equation*}
	and the set ${\mathcal C}(\e_0,\be)$ is defined by \eqref{dfcalCxe}-\eqref{dfCxef} with $\be\bot\e_0$ and $|\be|=1/2$. It us easy to check that for any orthogonal transformation $U$ of $\R^d$ with a positive Jacobian: $\tilde S(\e_0,\be)=\tilde S(U(\e_0),U(\be))$. This means that the quantity $\tilde S(\e_0,\be)$ does not depend on $\e_0$ and $\be$ satisfying the above conditions. Furthermore, we see from definition \eqref{dfcalCxe}-\eqref{dfCxef} that the interior of ${\mathcal C}(\e_0,\be)$ is not empty. Hence $\tilde S>0$. These circumstances and the left estimate of \eqref{estXxtyr}
	imply estimate \eqref{estXfrombelow} with $A(r_0)=(C(2r_0))^2\tilde S$.
	
	Let us prove estimate \eqref{estXfromabove}. In view of \eqref{dfSxtyr},
	\begin{equation*}
	 S(\e,\bt,\by,r)\le\frac{1}{|\e-\bt|^{d-2}}\bar S(\e_0),
	\end{equation*}
	 where 
	$\bar S(\e_0)=\int_{\R^d}\frac{\mathrm{d}\bu}{|\bu-\e_0|^{d-1}|\bu+\e_0|^{d-1}}$. It is easy to check that since $d\ge 3$, then $\bar S(\e_0)<\infty$. Furthermore, we see again that $\bar S(U(\e_0))=\bar S(\e_0)$ for any orthogonal transformation $U$ of $\R^d$ with a positive Jacobian, i.e., $\bar S(\e_0)$ does not depend on $\e_0$. Therefore the right estimate of \eqref{estXxtyr}
	imply estimate \eqref{estXfromabove} with $B(r_0)=(D(2r_0))^2\bar S$.
\end{proof}

We shall denote by $\mathrm{co}$ the convex hull of a set and by $\overline{\mathrm{co}}$ the closure of this convex hull. In the proof of Proposition \ref{estconvBeskern} we have used the following claim:

\begin{lemma}\label{lmconvhullofcones}
	Suppose that a closed ball $\bar B_R(\by_0)$ contains two points $\e_0$ and $-\e_0$ with $|\e_0|=1/2$. 
	Then there is $\be\in\R^d$ such that $\be\bot\e_0$, $|\be|=1/2$ and the set 
	\begin{equation}\label{dfcalCxe}
	{\mathcal C}(\e_0,\be)=\overline{\mathrm{co}}\Big(\bigcup_{\fb\bot\mathrm{span}\big(\{\e_0,\be\}\big)} C(\e_0,\be,\fb)\Big),
	\end{equation}
	where
	\begin{equation}\label{dfCxef}
	C(\e_0,\be,\fb)=\mathrm{co}\Big(\{\be\}\cup\{\bu\in\mathrm{span}\big(\{\e_0,\fb\}\big):\,|\bu|\le 1/2\}\Big),
	\end{equation}
	is contained in $\bar B_R(\by_0)$.
\end{lemma}
\begin{proof}
	Notice that since $\e_0,\,-\e_0\in\bar B_R(\by_0)$ and $|\e_0|=1/2$, then $R\ge 1/2$. Consider the one-dimensional subspace of $\R^d$: $L_0:=\mathrm{span}({\e_0})$ and denote by $P_1$ the operator of orthogonal projection on $L_0$. Denote $L_0^\bot:=\R^d\ominus L_0$. Consider the $d-1$-dimensional closed ball $\bar B_h^{d-1}(\bz_0):=\bar B_R(\by_0)\cap L_0^\bot$, where $\bz_0=(I-P_1)\by_0$. In the case where $\e_0$ and $\by_0$ are linearly independent (i.e., $\bz_0\neq \vec 0$) consider the two-dimensional subspace $\Pi:=\mathrm{span}({\e_0,\by_0})\subset\R^d$. In the case where $\e_0$ and $\by_0$ are linearly dependent (i.e., $\bz_0=\vec 0$), let us take as $\Pi$ any two-dimensional subspace of $\R^d$, containing $L_0$. Consider the disk $D:=\bar B_R(\by_0)\cap\Pi$ and the triangle $\vec a\,\vec b\,\vec c$ inscribed in $D$ with $\{\vec a,\,\vec c\}=\partial D\cap \big(L_0+\{\by_0\}\big)$ and the point $\vec b$ is belongs to the two-point set $\partial D\cap \bar B_h^{d-1}(\bz_0)$. In the case where $\bz_0\neq 0$ let us choose $\vec b$ such that $\vec b\cdot\bz_0<0$. Since the angle of this triangle at the vertex $\vec b$ is direct and the interval $[\vec a,\,\vec c]$ is orthogonal to the interval $[\bz_0,\,\vec b]$, we have: $h=|\bz_0-\vec b|=\sqrt{|\vec a-\bz_0||\vec c-\bz_0|}$. Furthermore, since $\e_0,\,-\e_0\in D$, the interval $[\e_0,\,-\e_0]$ is parallel to the interval $[\vec a,\,\vec c]$ and $|\e_0|=1/2$, then we see that $|\vec a-\bz_0|\ge1/2$ and $|\vec c-\bz_0|\ge1/2$. Hence $h\ge 1/2$. Let us take $\be_0=\frac{\bz_0}{2|\bz_0|}$, if $\bz_0\neq \vec 0$ and if $\bz_0=\vec 0$, let us take as $\be_0$ any vector $\be_0\in\Pi$ such that $\be_0\bot\be_0$ and $|\be_0|=1/2$. We see that in any case $\Pi=\mathrm{span}(\{\be_0,\,\e_0\})$. Since the ball $\bar B_R(\by_0)$ is a cosed convex set, then in in view of \eqref{dfcalCxe}, \eqref{dfCxef}, in order to prove the inclusion ${\mathcal C}(\e_0,\be)\subseteq\bar B_R(\by_0)$, we should show that for any $\fb\bot\Pi$ with $|\fb|=1$
	\begin{equation}\label{inclCB3}
	C(\e_0,\be,\fb)\subseteq\bar B_R^3(\by_0) =\bar B_R(\by_0)\cap\mathrm{span}\big(\{\e_0,\be_0,\fb\}\big).
	\end{equation}
	On the other hand, since the point $\vec 0$ belongs to the interval $[\bz_0,\,\vec b]$, $h\ge 1/2$ and $|\be|=1/2$, then $\be_0\in\bar B_h^{d-1}(\bz_0)$. Hence $\be_0\in\bar B_R(\by_0)$. Therefore in view of \eqref{dfCxef}, in order to prove inclusion \eqref{inclCB3}, we should show that the disk
	$\tilde D=\{\bu\in\tilde\Pi:\,|\bu|\le 1/2\}$ with $\tilde\Pi=\mathrm{span}(\{\e_0,\fb\})$ is contained in the disk $\hat D=\bar B_R^3(\by_0)\cap\tilde\Pi$. Since $\fb\bot\Pi$, we can consider in the plane $\tilde\Pi$ the coordinate system $XY$ with the orthonormal basis $\fb,\,2\e_0$. Then in these coordinates $\tilde D=\{(x,y):\,x^2+y^2\le 1/4\}$. Taking into account that the plane $\Pi$ contains the center $\by_0$ of the ball $bar B_R^3(\by_0)$, it is easy to see that in these coordinates $\hat D=\{(x,y):\,x^2+(y-a)^2\le r^2\}$, where $a=2\e_0\cdot\by_0$. Notice that since $\e_0,\,-\e_0\in \hat D$, then $r\ge 1/2$ and 
	\begin{equation}\label{inclinterv}
	[-\e_0,\,\e_0]\subseteq[\vec d,\,\vec g],
	\end{equation}
	where $\vec d$ and $\vec g$ are the points of intersection of the axis $Y$ with the circle $\partial\hat D$. It is clear that in order to prove that $\tilde D\subseteq\hat D$, it is enough to show that for any $x\in[-1/2.\,1/2]$
	\begin{equation}\label{ineqf12}
	f_1(x)\ge f_2(x) 
	\end{equation}
	and 
	\begin{equation}\label{ineqf34}
	f_3(x)\ge f_4(x),
	\end{equation}
	where $f_1(x)=\sqrt{r^2-x^2}+a$, $f_2(x)=\sqrt{1/4-x^2}$, $f_3(x)=-f_2(x)$ and $f_4(x)=-\sqrt{r^2-x^2}+a$. In view of \eqref{inclinterv}, $f_1(0)\ge f_2(0)$ and $f_3(0)\ge f_4(0)$. Furthermore, after the calculation we have: $f_1^\prime(0)=f_2^\prime(0)=f_3^\prime(0)=f_4^\prime(0)=0$ and $f_1^{\prime\prime}(x)=-\frac{1}{\sqrt{r^2-x^2}}-\frac{x^2}{(r^2-x^2)^{3/2}}$, $f_2^{\prime\prime}(x)=-\frac{1}{\sqrt{1/4-x^2}}-\frac{x^2}{(1/4-x^2)^{3/2}}$,
	$f_3^{\prime\prime}(x)=\frac{1}{\sqrt{1/4-x^2}}+\frac{x^2}{(1/4-x^2)^{3/2}}$,
	$f_4^{\prime\prime}(x)=\frac{1}{\sqrt{r^2-x^2}}+\frac{x^2}{(r^2-x^2)^{3/2}}$. Since $r\ge 1/2$, we see that for any $x\in[-1/2,\,1/2]$ $f_1^{\prime\prime}(x)\ge f_2^{\prime\prime}(x)$ and $f_3^{\prime\prime}(x)\ge f_4^{\prime\prime}(x)$.  Then by Lemma \ref{comparconvex}, inequalities  \eqref{ineqf12}, \eqref{ineqf34} are valid. 
\end{proof}

In the proof of Lemma \ref{lmconvhullofcones} we have used the following claim:

\begin{lemma}\label{comparconvex}
	Suppose that real-valued functions $f_1(x)$ and $f_2(x)$ are defined on an interval $[-r,r]\,(r>0)$ and belong to $C^2[-r,r]$. If $f_1(0)\ge f_2(0)$, $f_1^\prime(0)=f_2^\prime(0)=0$  and $f_1^{\prime\prime}(x)\ge f_x^{\prime\prime}(x)$ for any $x\in[-r,r]$,	then $f_1(x)\ge f_2(x)$ for any $x\in[-r,r]$.  
\end{lemma}
\begin{proof}
	Denote $\phi(x)=f_1(x)-f_2(x)$. Then from the obvious representation 
	\begin{displaymath}
	\phi(x)=\left\{\begin{array}{ll}
	\phi(0)+\int_0^x(x-t)\phi^{\prime\prime}(t)\,dt&\quad\mathrm{for}\quad x\in[0,r],\\
	\phi(0)+\int_x^0(t-x)\phi^{\prime\prime}(t)\,dt&\quad\mathrm{for}\quad x\in[-r,0)
	\end{array}\right.
	\end{displaymath}
	we obtain the desired claim.
\end{proof}

\section{Existence of the repeated nonincreasing rearrangement}
\label{sec:C}

	Let $(Y,\,\Xi,\,\nu)$ be a   measure spaces with  non-negative measure 
	$\nu$.
	Consider the Riesz space $M_\nu(Y)$ \cite{Lux-Za} of $\nu$-measurable functions $\phi:\,Y\rightarrow\R$, taking finite values for $\nu$-almost all $\bs\in Y$, with identification of $\nu$-almost all equal functions.  Recall that the lattice ordering relation 
	$v\ge w\,(v,w\in M_\nu(Y))$ means that $v(\bs)\ge w(\bs))$ for $\nu$-almost all $\bs\in Y$. Denote by $\bigvee(\mathcal{A})$ and $\bigwedge(\mathcal{A})$ the lattice supremum and infimum for a set of functions  $\mathcal{A}\subseteq M_\nu(Y)$.  It is known that if the measure $\nu$ finite, the space $M_\nu(Y)$ is {\it super Dedekind complete} (\cite{Lux-Za}, Chapt 4), i.e. it is {\it order separable} and Dedekind complete. Dedekind completeness means that every non-empty subset of $M_\nu(Y)$ which is bounded from above has a lattice supremum in $M_\nu(Y)$ and order separability means that every non-empty subset having a lattice supremum in $M_\nu(Y)$ contains at most countable subset having the same lattice supremum.  Furthermore, denote by $M_\nu^+(Y)$ the part of $M_\nu(Y)$, consisting of functions which are positive for $\nu$-almost all $\bs\in Y$. Along with the measure space $(Y,\,\Xi,\,\nu)$ consider another measure space $(X,\,\Sigma,\,\mu)$ with a non-negative measure $\mu$. Let $F:\,X\times Y\rightarrow\R$ be a nonnegative function, belonging to $L_1(X\times Y,\,\mu\times\nu)$.
	For a fixed $v(\cdot)\in M_\nu^+(Y)$ consider the function  
	\begin{equation}\label{dflambdXstar}
	\lambda_X^\star(v(\bs))=\mu\Big(\{\e\in X:\, F(\e,\bs)\ge v(\bs)\}\Big)=\int_X (h_{v(\bs)}\circ F)(\e,\bs)\mu(\mathrm{d}\e),
	\end{equation}
	where $h_{v(\bs)}(x)=h(x-v(\bs))$ and $h(x)$ is the Heaviside function,	and the set
	\begin{equation}\label{dfAt}
	\mathcal{A}_t=\{v\in M_\nu^+(Y):\,\lambda_X^\star(v(\bs))\ge t\}\quad (t>0).
	\end{equation}
	\begin{proposition}\label{prexistpartrear}
		Suppose that the measure $\nu$ is finite. Then
		
	(i) there exists a nonnegative function $F^\star(\cdot,\bs)(t;\, X;\,\mu)$,  defined by 
	\begin{equation}\label{dfpartrear}
	F^\star(\cdot,\bs)(t;\, X;\,\mu)=\bigvee\Big(\mathcal{A}_t\Big)
	\end{equation}
	and belonging to $L_1(Y,\,\nu)$; 
	
	(ii) for $\nu$-almost all $\bs\in Y$ the function $F^\star(\cdot,\bs)(t;\, X;\,\mu)$ is the nonincreasing rearrangement of $F(\e,\bs)$ by the variable $\e$, i.e., $F^\star(\cdot,\bs)(t;\, X;\,\mu)=\sup\{v\in\mathcal{A}_t:\,v(\bs)\}$. 
	\end{proposition}
\begin{proof}
(i) It is easy to see that  $h_{v(\cdot)}\circ F(\cdot,\cdot)\in L_1(X\times Y,\,\mu\times\nu)$, hence by \eqref{dflambdXstar} and Fubini's Theorem,  the function $\lambda_X^\star(v(\bs))$ belongs to $M_\nu(Y)$. 
Using \eqref{dfAt} and the Chebyshev inequality, we get that the inclusion $v\in\mathcal{A}_t$ implies:
\begin{equation}\label{Chebysh}
v(\bs)\le\frac{1}{t}\int_XF(\e,\bs)\mu(\mathrm{d}\e). 
\end{equation}
 Since the right hand side of the last inequality belongs to $M_\nu(Y)$, the set $\mathcal{A}_t$, defined by \eqref{dfAt}, is bounded from above in $M_\nu(Y)$.  Claim (i) of Lemma \ref{lmpontwisesupinf} implies that there exists in $M_\nu(Y)$ the function defined by \eqref{dfpartrear}.  Furthermore, since $F\in L_1(X\times Y,\,\mu\times\nu)$, the function in the right hand side of \eqref{Chebysh} belongs to $L_1(Y,\,\nu)$. These circumstances imply that $F^\star(\cdot,\bs)(t;\, X;\,\mu)\in L_1(Y,\,\nu)$. Claim (i) is proven.
 
 Claim (ii)	 follows immediately from claim (ii) of Lemma \ref{lmpontwisesupinf}.
\end{proof}
\begin{lemma}\label{lmpontwisesupinf}
	Suppose that the measure $\nu$ is finite.
	If a set $\mathcal{A}\subseteq M_\nu(Y)$ is bounded from above (from below), then
	
	(i) there exists in $M_\nu(Y)$ the lattice supremum $\bigvee(\mathcal{A})$ (the lattice infimum $\bigwedge(\mathcal{A})$) ;
	
	(ii) $\bigvee(\mathcal{A})$ ($\bigwedge(\mathcal{A})$) coincides $\nu$-almost everywhere in $Y$ with the pointwise supremum (infimum) of $\mathcal{A}$.
	
\end{lemma}
\begin{proof} Claim (i) follows from the Dedekind completeness of $M_\nu(Y)$.
	
 (ii) It is enough consider the case where $\mathcal{A}$ is bounded from above.  Denote $\bar a=\bigvee\Big(\mathcal{A}\Big)$. Since $v\le \bar a$ for any $v\in\mathcal{A}$, then $\bar v(\bs)=\sup_{v\in\mathcal{A}}\le\bar a(\bs)$ for $\nu$-almost all $\bs\in Y$. Let us prove the inverse inequality. Since $M_\nu(Y)$ is order separable, then there is an at most countable subset $\mathcal{A}_{\sigma,t}\subseteq\mathcal{A}$ such that $\bar a_\sigma=\bigvee\Big(\mathcal{A}_{\sigma}\Big)=\bar a$. Denote by $\bar v_\sigma(\bs)$ the pointwise supremum of $\mathcal{A}_{\sigma}$. Then $\bar a(\bs)\ge\bar v(\bs)\ge\bar v_\sigma(\bs)$ for $\nu$-almost all $\bs\in Y$. Hence the function $\bar v_\sigma(\bs)$ is $\nu$-almost everywhere finite. On the other hand, it is known that the pointwise supremum of a countable set of $\nu$-measurable functions is $\nu$-measurable (\cite{Hal}, Sec. 20).   Therefore $\bar v_\sigma\in M_\nu(Y)$. Furthermore, it is clear that $\bar v_\sigma$ is an upper bound for $\mathcal{A}_{\sigma}$ in $M_\nu(Y)$, hence 
$\bar v_\sigma\ge\bar a_\sigma=\bar a$. Therefore $\bar v(\bs)\ge\bar a(\bs)$ for $\nu$-almost all $\bs\in Y$. Thus, we have shown that $\bar v(\bs)=\bar a(\bs)$ for $\nu$-almost all $\bs\in Y$.	
\end{proof}

	\begin{definition}\label{dfreprearr}
	We call the function $F^\star(\cdot,\bs)(t;\, X;\mu)$, defined by \eqref{dfpartrear}, the {\it partial  nonincreasing rearrangement} of $F(\e,\bs)$ by the variable $\e$.
	Then by claim (i) of Proposition \ref{prexistpartrear}, for $u>0$ the non-increasing rearrangement
	\begin{equation*}
	(F^\star(\cdot,\cdot)(t;\,X;\,\mu))^\star(u;\,Y;\,\nu)
	\end{equation*}  
	 of the previous function by the variable $\bs$ has a sense. We call it the {\it repeated non-increasing rearrangement} of $F(\e,\bs)$ and denote it briefly by
	 \begin{equation*}
	 (F^\star)^\star(t,u;\,X,Y;\,\mu,\nu).
	 \end{equation*}
	
\end{definition}


\begin{thebibliography}{ABC2}
	\addcontentsline{toc}{section}{Bibliography}
	
	\bibitem[AH]{AH} D.R. Adams and L.I. Hedberg, \textit{Function Spaces and Potential
	Theory,} Grundlehren \textbf{314}, Springer-Verlag, 1996.

	\bibitem[Ad]{Ad} D.R. Adams, \textit{Choquet integrals in potential
		theory,} Publicacions Matematiques, Vol \textbf{42} (1998), 3-66.
	
	\bibitem[AF]{AF} R.A. Adams and J.J.F. Fournier, \textit{Sobolev Spaces,} Second Edition, Academic Press (Elsevier), 2003.





\bibitem[Ben-Fort]{Ben-Fort} V. Benci and D. Fortunato, \textit{Discreteness  Conditions of the Spectrum of Schr\"odinger
	Operators,} J. Math. Anal. Appl., \textbf{64} (1978), 695-700.





\bibitem[Ch]{Ch} G. Choquet, \textit{Theory of Capacities,}
Ann. Inst. Fourier, \textbf{5} (1953), 131-295.





\bibitem[GMD]{GMD} Gian Maria Dall'Ara, \textit{Discreteness of the spectrum of Schr\"odinger operators with matrix-valued non-negative
	potentials,} Journal of Functional Analysis \textbf{268} (2015),
3649-3679.


\bibitem[Fuj]{Fuj} S. Fujishige, \textit{On the subdifferential of a submodular
	function,} Mathematical Programming \textbf{29} (1984), 348-360.




\bibitem[Hal]{Hal} P.R. Halmos, \textit{Measure Theory,} Springer-Verlag New York Heidelberg
Berlin, 1950.	







\bibitem[John]{John} R.A. Johnson, \textit{Atomic and non-atomic measures,} Proc. Amer. Math. Soc. \textbf{25} (1970), 650-655. 

\bibitem[L-S-W]{L-S-W} D. Lenz, P. Stollmann and D. Wingert, \textit{Compactness of Schr\"odinger semigroups,} Math. Nachr.
\textbf{283} (2010), No 1, 94-103.

\bibitem[Lux-Za]{Lux-Za}  W.A.J. Luxemburg and A.C. Zaanen, \textit{Riesz Spaces,} North-Holland Publishing Company, Amsterdam-London, 1971.


\bibitem[Mar-Mon]{Mar-Mon} M. Marinacci and L. Montrucchio,
\textit{A characterization of the core of convex games through
	Gateaux derivatives,} Journal of Economic Theory, \textbf{116}
(2004), 229-248.

\bibitem[Maz1]{Maz1} V. Mazya, \textit{Sobolev Spaces: with Applications to Elliptic Partial Differential Equations,} Springer, 2nd edition, series: Grundlehren der mathematischen Wissenschaften 342, 2011.

\bibitem[Maz]{Maz} V. Mazya, \textit{Analytic criteria in the qualitative spectral analysis of the Schr\"odinger
	operator,} Procieedings of Simposia in Pure Mathematics, Spectral Theory and Mathematical Physics: A Festschrift in Honor of Barry Simon's 60th Birthday, A.M.S., Providence, Rhode Island, Vol {\bf 76}, Part {\bf 1} (2006), 257-288.

\bibitem[M-Sh]{M-Sh} V. Mazya and M. Shubin, \textit{Discreteness of spectrum and positivity criteria for Schr\"odinger operators,} Ann. Math., \textbf{162} (2005), 919-942.









\bibitem[Schm]{Schm} D. Schmeidler, \textit{Integral representation without
	additivity,} Proc. Amet. Math. Soc. Vol \textbf{97} (1986), No 2,
255-261.


\bibitem[Shap]{Shap} L.S. Shapley, \textit{Cores of convex games,} Internat J. Game Theory \textbf{1} (1971), 11-26.	


\bibitem[Si1]{Si1} B. Simon,  \textit{Schr\"odinger operators with purely discrete
	spectrum,}  Methods of Functional Analysis and Topology, Vol.
\textbf{15} (2009), no. 1, 61-66.

\bibitem[T]{T}  M. Taylor \textit{Scattewring Lengtth and the Spectrum of $-\Delta+V$,} Canad. Math. Bull. Vol \textbf{49} (1) (2006), 144-151.

\bibitem[Zel1]{Zel1} L. Zelenko, \textit{Conditions of discreteness of the spectrum for Schr\"odinger operator and some optimization problems for capacity and measures,} Applied Analysis and Optimization, special issue ``Nonlinear Analysis and Optimization'', dedicated to Professor Yakov Alber on the occasion of his 80th birthday- to appear. Preprint: ArXiv: 1812.00416






	
\end{thebibliography}
\end{document}